\documentclass{amsart}
\usepackage{latexsym}
\usepackage{amssymb}
\usepackage{amsmath}
\usepackage{amsfonts}
\usepackage{hyperref}
\usepackage{enumerate}

\def\R{\mathbb{R}}

\newcommand{\bmat}{\left[\begin{matrix}}
\newcommand{\emat}{\end{matrix}\right]}
\usepackage[latin1]{inputenc} 
\usepackage{amsthm}
\usepackage{amsxtra}
\usepackage{amscd}
\usepackage{setspace}   
\usepackage{mathrsfs}   

\numberwithin{equation}{section}
\numberwithin{table}{section}
\numberwithin{figure}{section}
\newtheorem{theorem}{Theorem}[section]

\newtheorem{lemma}{Lemma}[section]

\newtheorem{corollary}{Corollary}[section]

\theoremstyle{remark}
\newtheorem*{remark}{Remark}

\theoremstyle{definition}

\newcommand{\Z}{\mathbb{Z}}
\newcommand{\Q}{\mathbb{Q}}

\newcommand{\Gr}{\mathrm{Gr}}
\newcommand{\GL}{\mathrm{GL}}
\newcommand{\SL}{\mathrm{SL}}

\newcommand{\rank}{\mathrm{rank\,}}
\newcommand{\rk}{\mathrm{rk\,}}

\newcommand{\spn}{\mathrm{span}}

\newcommand{\tr}{\mathrm{tr}}
\newcommand{\lat}{\mathcal}

\makeatletter
\@namedef{subjclassname@2020}{%
  \textup{2020} Mathematics Subject Classification}
\makeatother

    \title{Counting rational points of a Grassmannian}
    \author{Seungki Kim}

\begin{document}

\maketitle

\begin{abstract}
We prove an estimate on the number of rational points on the Grassmannian variety of bounded twisted height, refining the classical results of Schmidt (\cite{Sch}) and Thunder (\cite{Thu2}) over the rational field: most importantly, our formula counts all points. Among the consequences are a couple of new implications on the classical subject of counting rational points on flag varieties.
\end{abstract}

\section{Introduction}

\subsection{Main result}

For $\lat L \subseteq \R^n$ a lattice, $1 \leq d < \rk \lat L$ and $H > 0$, let $P(\lat{L}, d, H)$ be the number of primitive rank $d$ sublattices of $\lat{L}$ of determinant less than or equal to $H$. The purpose of this paper is to investigate the quantitative behavior of $P(\lat{L}, d, H)$. The earliest result of this kind goes back to the mid-twentieth century, due to W. Schmidt (\cite{Sch}):

\begin{theorem}[Schmidt \cite{Sch}, Theorem 1] \label{thm:schmidt}
Let
\begin{align*}
a(n, d) &= \frac{1}{n}\binom{n}{d} \prod_{i=1}^{d}\frac{V(n-i+1)}{V(i)} \cdot \frac{\zeta(i)}{\zeta(n-i+1)}, \\
b(n, d) &= \max\left(\frac{1}{d}, \frac{1}{n-d}\right),
\end{align*}
where $V(i) := \pi^{i/2}/\Gamma(i/2+1)$ is the volume of the unit ball in $\R^i$ and $\zeta(s)$ is the Riemann zeta function, except that we understand $\zeta(1)=1$ for convenience. Then
\begin{equation*}
P(\Z^n, d, H) = a(n, d)H^n + O(H^{n-b(n,d)}),
\end{equation*}
where the implicit constant depends on $n$ only. 
\end{theorem}

For $\lat L \subseteq \R^n$ of full rank, $P(\lat{L}, d, H)$ may also be understood in terms of a counting problem on the Grassmannian variety $\Gr(n,d)$ consisting of the $d$-dimensional subspaces of $\R^n$. A rational point $P$ on $\Gr(n,d)$ is a $d$-dimensional subspace such that $P \cap \Z^n$ is a rank $d$ sublattice of $\Z^n$. Its height is given by the determinant of $P \cap \Z^n$, and given a real $n \times n$ matrix $L$, its height twisted by $L$ is given by the determinant of the rank $d$ lattice
\begin{equation*}
(P \cap \Z^n) L = \{vL : v \in P \cap \Z^n\} \subseteq \R^n
\end{equation*}
(we write vectors horizontally, so that matrices multiply from the right). Observe that, if $L$ and $L'$ are two $n \times n$ matrices whose row vectors generate the same lattice $\lat L$, i.e. $L = gL'$ for some $g \in \GL(n,\Z)$, then the number of the rational points whose heights are bounded by $H$ is the same, regardless of whether one twists the height by $L$ or $L'$. In addition, one checks that this number is precisely $P(\lat{L}, d, H)$ as defined above. We refer the reader to Thunder (\cite[Introduction]{Thu1}, \cite[Part I]{Thu2}) for the general definition of a twisted height, in which it is first introduced.

From this perspective, Thunder (\cite{Thu2}) proved a vast generalization of Theorem \ref{thm:schmidt} above, extending it to any lattice and to any number field $K$ (here $\mathcal{O}_K$-modules play the role of lattices). His result, from the 1990's, remains state-of-the-art to this day. We state his result in case $K = \Q$:

\begin{theorem}[Thunder \cite{Thu2}, Theorem 3] \label{thm:thunder}
Let $\lat L \subseteq \R^n$ be a lattice of full rank. In addition to the notations in Theorem \ref{thm:schmidt}, define
\begin{equation*}
\lat L_{i} = \lat L \cap \spn_\R(v_1, \ldots, v_{i}),
\end{equation*}
where $v_1, \ldots, v_n$ are choices of linearly independent vectors in $\lat L$ such that  $\|v_i\| = \lambda_i(\lat L)$, where $\lambda_i(\lat L)$ is the $i$-th successive minimum of $\lat L$ defined by
\begin{equation*}
\lambda_i(\lat L) = \inf\{r > 0: \dim \spn_\R(\lat L \cap B(r)) \geq i \},
\end{equation*}
where $B(r)$ here is the closed ball at origin of radius $r$ in $\R^n$. Let $P_{\lat L_{n-d}}(\lat L, d, H)$ be the number of rank $d$ sublattices of $\lat L$ of determinant $\leq H$ whose intersection with $\lat L_{n-d}$ is trivial. Then

\begin{equation*}
P_{\lat{L}_{n-d}}(\lat L, d, H) = a(n,d)\frac{H^n}{(\det \lat L)^d} + O\left(\frac{H^{n-b(n,d)}}{(\det \lat L)^{d - b(n,d)}(\det \lat L_{n-d})^{b(n,d)}}\right),
\end{equation*}
where the implicit constant depends only on $n$.

\end{theorem}

A notable feature of Theorem \ref{thm:thunder} is that it provides an explicit description of the dependence of the error term on the successive minima $\lambda_1(\lat L), \ldots, \lambda_n(\lat L)$ of $\lat L$ (observe that $\det \lat L_i \sim \lambda_1\lambda_2 \ldots \lambda_i$ by the Minkowski's second theorem). Informally speaking, it reflects the ``skewness'' of the lattice: in case $\lat L$ is severely skewed, in the sense that $\lambda_i(\lat L)$ is much smaller than $\lambda_{i+1}(\lat L)$ for some $i$, one expects a different behavior of the error term than the case in which most $\lambda_i(\lat L)$ are about equal. Theorem \ref{thm:thunder} may be seen as a realization of this intuition. 

However, Thunder (\cite{Thu2}) does not provide a corresponding estimate for $P(\lat L, d, H)$, remarking that it would ``be a cumbersome task.'' In the present paper, we introduce a method that circumvents this difficulty, and prove

\newtheorem{innercustomthm}{Theorem}
\newenvironment{theorem'}[1]
  {\renewcommand\theinnercustomthm{#1}\innercustomthm}
  {\endinnercustomthm}

\begin{theorem'}{1.3'}\label{thm:main'}
Continue with the notations in Theorems \ref{thm:schmidt} and \ref{thm:thunder} above. Then for all $H > 0$,
\begin{equation} \label{eq:main'}
P(\lat L, d, H) = a(n, d)\frac{H^n}{(\det \lat L)^d} + O\left(1 + \left(\frac{H}{\lambda_1(\lat L)^d}\right)^{n-b(n,d)} \right),
\end{equation}
where the implied constant depends only on $n$.
\end{theorem'}

In fact, we prove the more precise

\begin{theorem}\label{thm:main}
Let $\lat L \subseteq \R^n$ be a lattice of full rank. For all $H > 0$,
\begin{equation} \label{eq:main}
P(\lat L, d, H) = a(n, d)\frac{H^n}{(\det \lat L)^d} + O\left(\sum_{j \in E_{n,d}} b_j(\lat L)H^{\gamma_j}\right),
\end{equation}
where the implied constant in the big-O notation depends only on $n$, and $E_{n,d}$ is a finite set of indices, of cardinality less than $n^{3n}$ for $n \geq 2$ . Each $j \in E_{n,d}$ is associated with $\gamma_j \in [0, n- b(n,d)]$ and $b_j$ of the form
\begin{equation*}
b_j(\lat L) = \prod_{i=1}^n (\det \lat L_i)^{-\beta_i}
\end{equation*}
for some real $\beta_i \geq 0$ such that $\sum_{i=1}^n i\beta_i = d\gamma_j$. This makes the right-hand side of \eqref{eq:main} scale-invariant i.e. it remains unchanged if $\lat L$ and $H$ are replaced by $c\lat L$ and $c^dH$, respectively.

In particular, the leading error term is
\begin{equation*}
\frac{H^{n-b(n,d)}}{(\det \lat L)^{d - b(n,d)}(\det \lat L_{n-d})^{b(n,d)}},
\end{equation*}
as in Theorem \ref{thm:thunder}.

\end{theorem}

\begin{remark}
If $\lat L$ is of rank $m < n$, we may adapt Theorem \ref{thm:main} by observing that, for any isometry $\iota: \R^m \rightarrow \spn_\R(\lat L)$, it holds that $P(\lat L, d, H) = P(\iota^{-1}(\lat L), d, H)$. The same applies to the results of the similar flavor that we state below.
\end{remark}

\begin{remark}
The estimate $|E_{n,d}| < n^{3n}$ is an extremely crude one, provided only to assure that the sum is finite. Describing $E_{n,d}$ exactly from our computations would be quite a laborious task that would not yield a pretty formula and whose fruits seem unclear. 

Thus one may feel that the complicated statement of Theorem \ref{thm:main} is unnecessary. However, we leave it as it is, since the precise knowledge of at least a part of the error term may be useful for certain applications. For instance, to prove Corollary \ref{cor:3} below, we really need the (close-to-)optimal version of the leading error term stated as in Theorem \ref{thm:main}. If we could compute the coefficients $b_j$'s optimally for more $j$'s, we expect to be able to strengthen Corollary \ref{cor:3} accordingly.

We also state the subsequent results in the precise form. If one desires simplification, one may replace $b_j$ by an appropriate power of $\lambda_1$'s, as in \eqref{eq:main'}.
\end{remark}

Before we go on to discuss a few applications of Theorem \ref{thm:main}, let us present a few of its variants that may also be of use.

\begin{theorem}\label{cor:1}
Let $N(\lat L, d, H)$ be the number of (not necessarily primitive) rank $d$ sublattices of $\lat L$ of determinant $\leq H$. Also let
\begin{equation*}
c(n, d) = a(n,d)\prod_{i=1}^{d} \zeta(n-i+1).
\end{equation*}

Then similarly to Theorem \ref{thm:main}, for a full-rank lattice $\lat L \subseteq \R^n$ we have
\begin{equation*}
N(\lat L, d, H) = c(n,d)\frac{H^n}{(\det \lat L)^d} + O\left(\sum_{j \in F_{n,d}} b'_j(\lat L)H^{\gamma_j}\right),
\end{equation*}
where the implied constant in the big-O notation depends only on $n$, and $F_{n,d} $ is a set of indices of cardinality at most $n^{3n}$. The description of $b'_j$ (resp. $\gamma_j$) is the same as that of $b_j$ (resp. $\gamma_j$) in Theorem \ref{thm:main}. Also, if $d < n-1$, the leading error term is the same as in Theorem \ref{thm:main}. If $d=n-1$, then the largest $\gamma_j$ for $j \in F_{n,d}$ may be chosen to be $n-1+\eta$, for any $\eta > 0$.
\end{theorem}

\begin{theorem}\label{cor:2}
For a lattice $\lat L \subseteq \R^n$, choose a primitive sublattice $\lat S \subseteq \lat L$ of rank $\leq n-d$, and let $P_\lat{S}(\lat L, d, H)$ be the number of primitive rank $d$ sublattices of $\lat L$ whose intersection with $\lat S$ is trivial. For $\lat L \subseteq \R^n$ of full rank, we have
\begin{equation*}
P_\lat{S}(\lat L, d, H) = a(n, d)\frac{H^n}{(\det \lat L)^d} + O\left(\sum_{j \in G_{n,d}} b_j(\lat L)H^{\gamma_j}\right),
\end{equation*}
where again the description of the error term is identical to that of Theorem \ref{thm:main}, except that $G_{n,d}$ is now a set of cardinality less than $3n^{3n}$. In particular, the leading error term is the same as in Theorem \ref{thm:main}. Moreover, this formula is independent of $\lat S$.

The analogous statement holds for $N_\lat{S}(\lat L, d, H)$, which counts both primitive and non-primitive lattices.
\end{theorem}

\subsection{Applications}


Below we demonstrate a few immediate applications of Theorem \ref{thm:main} and the techniques used in its proof. Its main strength lies in the fact that it counts all the sublattices, and that it provides information regardless of how skewed the given lattice is, in particular relative to $H$. In comparison, its precedent Theorem \ref{thm:thunder} misses the sublattices that nontrivially intersect $\lat L_{n-d}$, and thus it does not say anything about the lattices for which $\det \lat L / \det \lat L_{n-d} > H$.

We expect there to be more uses of Theorem \ref{thm:main}; for instance, see a recent work of Le Boudec (\cite{leb}), which employs the $d=1$ case (due to Schmidt (\cite{Sch}), see \eqref{eq:case1} below) as the main device.

\subsubsection{Rational points of flag varieties}

It is natural to expect that a counting formula on Grassmannians should yield a counting formula for general flag varieties. Indeed, Thunder (\cite{Thu2}) derives such a formula as a relatively simple application of Theorem \ref{thm:thunder}. We present its simplest case to initiate the discussion:

\begin{theorem}[Thunder \cite{Thu2}, Theorem 5]\label{cor:3thu}
Let $\lat L \subseteq \R^n$ be a lattice of rank $n$, and suppose $H$ is sufficiently large. Then the number of flags $\lat S^e \subseteq \lat S^d \subseteq \lat L$ of type $(e,d)$ (hence $\rk \lat S^i = i$) such that $\lat S^i \cap \lat L_{n-i} = 0$ for $i \in \{e, d\}$, and $\lat S^e \cap \lat S^d_{d-e} = 0$, whose height twisted by $\lat L$ is at most $H$ is
\begin{equation}\label{eq:l=2thu}
\frac{aH}{(\det \lat L)^d}\log\frac{H}{(\det \lat L)^d} + O\left(\frac{H}{(\det \lat L)^{d-b(n,d)(n-d)/n}(\det \lat L_{n-d})^{b(n,d)}}\right),
\end{equation}
where $a$ is some explicit constant depending only on $n, d, e$, and the implicit constant in the error depends only on $n$.

\end{theorem}

In this context, the height of the flag $\lat S^e \subseteq \lat S^d$ is the quantity $(\det \lat S^e)^d(\det \lat S^d)^{n-e}$; see e.g. Thunder (\cite{Thu2}) for details.

In comparison, we can derive from Theorem \ref{thm:main} the following

\begin{corollary} \label{cor:3}
Let $\lat L$ be a lattice of rank $n \geq 3$, and $H$ be sufficiently large --- more precisely, $\lambda_n(\lat L) \ll_n H^{1/nd}$. Then the number of flags $\lat S^e \subseteq \lat S^d \subseteq \lat L$ of type $(e,d)$ such that $\lat S^e \cap \lat S^d_{d-e} = 0$ whose height twisted by $\lat L$ is at most $H$ is
\begin{equation}\label{eq:l=2me}
\frac{aH}{(\det \lat L)^d} \log \frac{H}{\varepsilon_e^d \varepsilon_d^{n-e}} + O\left( \sum_{j \in E_{n,e,d}} b_j(\lat L) H^{\gamma_j} \right),
\end{equation}
where $\varepsilon_e = \min \{\det \lat X: \lat X \subseteq \lat L, \rk \lat X = e\}$ and likewise for $\varepsilon_d$, $a$ is the same as in \eqref{eq:l=2thu}, the implicit constant depends on $n$ only, $E_{n,e,d}$ is an index set of cardinality at most $3n^{3n+1}$, and $b_j(\lat L)$ and $\gamma_j$ are as in the statement of Theorem \ref{thm:main}, except that $\sum_i i\beta_i = dn\gamma_j$.

Furthermore, the largest $\gamma_j$ is $1$, and the next largest is either
\begin{equation*}
1-\frac{b(n,d)}{n},\  1-\frac{b(d,e)(n-e)}{nd},\  \mathrm{or}\  1-\frac{1}{n}\left(1-\frac{2b(d,e)}{d} + \frac{1-b(d,e)}{n-e}\right);
\end{equation*}
if $d \leq n/2$, it is always $1-b(n,d)/n = 1-1/dn$.

\end{corollary}

In order to keep the proof relatively short and simple, we had to keep some of the assumptions made by Theorem \ref{cor:3thu}. Still, it has a couple of new implications that may be of interest. First, it shows that there must exist a gap between Theorem \ref{cor:3thu} and an ideal counting formula that would count all the rational flags, and that it must be at least of size
\begin{equation*}
O\left(H\log \frac{(\det \lat L)^d)}{\varepsilon_e^d\varepsilon_d^{n-e}}\right).
\end{equation*}
For a heavily skewed $\lat L$, for instance if $\det \lat L = 1$ but $\varepsilon_d \ll H^{-1/n}$, then this is of comparable size to the main term.



The second implication has to do with the error term in the well-known theorem of Franke, Manin, and Tschinkel (\cite{FMT}) on the number of rational points on flag varieties. In the corollary to Theorem 5 therein, which says that the number of rational points on a flag variety $V$ of (``untwisted'') height bounded by $H$ is
\begin{equation*}
Hp(\log H) + o(H),
\end{equation*}
where $p$ is a polynomial of degree $\rk \mathrm{Pic}(V) - 1$, they conjecture that the error term is of size $O(H^{1-\varepsilon})$ with $\varepsilon = (\dim V)^{-1}$. On the other hand, when $V$ is a Grassmannian, the literature (for instance Schmidt \cite{Sch}, and Thunder \cite{Thu1} \cite{Thu2}) suggests that $\varepsilon = b(n,d)/n$, as their analyses seem fairly sharp. Our Corollary \ref{cor:3} extends this to flag varieties of type $(e,d)$, suggesting that we have $\varepsilon = b(n,d)/n$ again, at least when $d \leq n/2$. In case $e \geq n/2$, one may be able to estimate $\varepsilon = 1/(n-e)n$ by duality. But in general the nature of $\varepsilon$ appears rather complicated.

It is possible to prove by a similar argument an analogue of Corollary \ref{cor:3} for a flag variety of any type that is strong enough to yield these same implications. On the other hand, we expect the ideal formula that would count all points on a flag variety to require another substantial amount of effort along the lines of the present paper. As stated in the remark after Theorem \ref{thm:main}, if we could find explicit expressions for every $b_j$ in \eqref{eq:main}, preferably containing large powers of $\det \lat L$, it would allow our error-bounding techniques to apply immediately. Using the methods of the present paper, it may be possible to achieve this for the first few small values of $d$.

\subsubsection{Mean value theorems over lattices}

Mean value theorems over random lattices provide a method of averaging lattice-point counting formulas over the space $X_n := \SL(n,\Z) \backslash \SL(n, \R)$ of determinant $1$ full-rank lattices in $\R^n$. The first known such theorem is the famous Siegel integration formula:
\begin{theorem}[Siegel \cite{Sie}, Theorem on p.341] \label{thm:sie}
Let $f: \R^n \rightarrow \R$ be a Borel measurable and compactly supported function. Then
\begin{equation} \label{eq:sie}
\int_{X_n} \sum_{v \in \lat L \backslash \{0\}} f(v) d\mu_n(\lat L) = \int_{\R^n} f(x) dx,
\end{equation}
where $d\mu_n$ is the normalized Haar measure on $\SL(n,\R)$, and $dx$ is the Lebesgue measure.
\end{theorem}

For example, if one sets $f$ to be the characteristic function of a set $S \subseteq \R^n$, then the sum on the left-hand side of \eqref{eq:sie} equals $\left| S \cap \lat L \backslash \{0\} \right|$, and thus Theorem \ref{thm:sie} implies that a random lattice sampled according to $d\mu_n$ has on average $\mathrm{vol}(S)$ nonzero vectors contained in $S$. Another important example is the Rogers integral formula, one of the main tools in geometry of numbers (see e.g. \cite{Kim}, \cite{SW}, \cite{SS} for some of the applications):
\begin{theorem}[Rogers \cite{Rogers}, (essentially) Theorem 4] \label{thm:rogers}
For $k < n$, let $f: (\R^n)^k \rightarrow \R$ be a Borel measurable and compactly supported function. Then
\begin{equation*}
\int_{X_n} \sum_{v_1, \ldots, v_k \in \lat L \atop \mathrm{independent}} f(v_1, \ldots, v_k) d\mu_n(\lat L) = \int_{\R^n} \ldots \int_{\R^n} f(x_1, \ldots, x_k) dx_1\ldots dx_k,
\end{equation*}
where $d\mu_n$  is the normalized Haar measure on $\SL(n,\R)$, and each $dx_i$ is the Lebesgue measure.
\end{theorem}

In the author's work \cite{Kim2}, written concurrently with the present paper, the machinery that turns a lattice-point counting formula, such as Theorems \ref{thm:main} and \ref{cor:2}, into a mean-value theorem has been developed, inspired by the argument of Rogers (\cite{Rogers}).
As a result, the following extension of Theorem \ref{thm:rogers} to Grassmannians is obtained from Theorem \ref{cor:2}
\begin{theorem}[Kim \cite{Kim2}, Theorem 3] \label{thm:main_mean}
Suppose $1 \leq k < n$, $1 \leq d_1, \ldots, d_k < n$ and $d_1 + \ldots + d_k < n$. Define $P(\lat L, d_1, \ldots, d_k, H_1, \ldots, H_k)$ to be the number of independent primitive sublattices $\lat A_1, \ldots, \lat A_k$ of $\lat L$ of ranks $d_1, \ldots, d_k$ and determinants bounded by $H_1, \ldots, H_k$, respectively. Then
\begin{equation*}
\int_{X_n} P(\lat L, d_1, \ldots, d_k, H_1, \ldots, H_k) d\mu_n(\lat L) = \prod_{i=1}^k a(n,d_i)H_i^n.
\end{equation*}
\end{theorem}

We note that Thunder (\cite{Thu3}) proved the $k=1$ case of this result. Theorem \ref{thm:main_mean} has several implications on the statistics of the randomized heights of the points on Grassmannians and flag varieties, such as the following.

\begin{corollary}[Kim \cite{Kim2}, Corollary to Theorem 5] \label{cor:meanflag}
Let $k \geq 2$, and $1 = d_0 \leq d_1 < \ldots < d_k < d_{k+1} = n$. For a lattice $\lat L \subseteq \R^n$ and its flag $\lat S^{d_1} \subseteq \ldots \subseteq S^{d_k} \subseteq \lat L$ of type $\mathfrak{d} = (d_1, \ldots, d_k)$, its height is defined as the quantity
\begin{equation*}
\prod_{i=1}^k (\det \lat S^i)^{d_{i+1} - d_{i-1}}.
\end{equation*}
Let $P(\lat L, \mathfrak{d}, H)$ be the number of type $\mathfrak d$ flags of $\lat L$ whose heights are bounded by $H$. Then
\begin{equation*}
\int_{X_n} P(\lat L, \mathfrak{d}, H) d\mu_n(\lat L) = \infty.
\end{equation*}
\end{corollary}

It may be interesting to compare this result with Theorem \ref{cor:3thu} and Corollary \ref{cor:3}. The author speculates that the divergence here is related to the main term of \eqref{eq:l=2me} in the statement of Corollary \ref{cor:3} being dependent on the skewness of the lattice in question.

\subsection{Method of proof}

All previous works on this topic (\cite{Sch}, \cite{Thu1}, \cite{Thu2}) count ``upwards,'' i.e. they construct the $d$-dimensional sublattice from either a $(d-1)$-dimensional sublattice or a $d$-dimensional sublattice lying in an $(n-1)$-dimensional ambient space. Our main idea is to take the dual approach, and count ``downwards'' instead: we project all the $d$-dimensional sublattices to a hyperplane, and count the cardinality of each fiber. This lets us bypass some of the technical difficulties that arise when counting upwards.

To elaborate, we prove Theorem \ref{thm:main} by the following inductive procedure that resembles the Pascal's triangle method of computing the binomial coefficients. In case $d=1$ or $d=n-1$, the formulas are well-known. Otherwise, let $\bar{\lat L}$ be the projection of $\lat L$ onto the orthogonal complement of a shortest nonzero vector of $\lat L$. Then we have
\begin{equation} \label{eq:outline}
P(\lat L, d, H) = P(\bar{\lat L}, d-1, \frac{H}{\lambda_1(\lat L)}) + \Phi(P(\bar{\lat L}, d, H)),
\end{equation}
where $\Phi$ can be regarded as a certain integral transformation. For a choice of a basis $\{v_1, \ldots, v_n\}$ of $\lat L$ and a sublattice $\lat B \subseteq \lat L$ of rank $d$, let us say $\lat B$ is of d-type $(\alpha_1, \ldots, \alpha_n)$ --- ``d'' stands for ``dual'' --- if the projection of $\lat B$ onto $\spn_\R(v_1, \ldots, v_{n-i+1})^\perp$ has rank $\alpha_i$. Then the first term on the right-hand side of \eqref{eq:outline} is counting the sublattices of d-types $(*, \ldots, *, d-1, d)$, and the second term is counting those of d-types $(*, \ldots, *, d, d)$.

In comparison, Theorem \ref{thm:thunder} counts precisely the sublattices of d-type $(1, 2, \ldots, d, \ldots, d)$. The upward counting method forces one to count the sublattices of each d-type separately, which is precisely what Thunder refers to as being ``cumbersome.'' The downward method resolves this difficulty.

Most of this paper is devoted to explicitly writing out and estimating $\Phi(P(\bar{\lat L}, d, H))$. Many parts of the computation can be done by slightly refining the methods of Schmidt (\cite{Sch}) or Thunder (\cite{Thu2}). However, the fact that $\lat L$ can be arbitrarily skewed presents a new difficulty, especially when bounding the error terms. This is resolved by comparing the gaps between the successive minima to $H^{1/d}$: if $\lambda_{i+1} - \lambda_i \ll H^{1/d}$ for all $i=1, \ldots, n-1$, the lattice may be considered not so severely skewed, as the classical techniques continue to apply. If in contrast $\lambda_{i+1} - \lambda_i \gg H^{1/d}$ for some $i$, we exploit this gap to finesse the desired error bound.

\subsection{Organization}

In Section 2, we introduce the definitions and notations used throughout the paper, and state the known formulas for $P(\lat L, 1, H)$ and $P(\lat L, n-1, H)$. In Section 3, we set up the induction argument, establishing the precise version of \eqref{eq:outline}. Sections 4 and 5 are devoted to the main and error term estimates, respectively. Section 6 collects all the computations and concludes the proof of Theorem \ref{thm:main}. The variants are all proved in Section 7.

\subsection{Acknowledgment}

Part of this work is supported by NSF grant CNS-2034176. The author thanks the anonymous referee, Lillian Pierce, Anders S\"odergren, and Jeffery Thunder for helpful comments and suggestions.

\section{Some backgrounds}

\subsection{Definitions, notations, and conventions}

Unless mentioned otherwise, the definitions and notations of this section apply.

\subsubsection*{Generalities}
The lowercase letter $p$ denotes a prime. Let us write $\Gamma = \GL(d, \Z)$ for short.
 We use capital letters such as $L, M$ to refer to matrices, and calligraphic fonts such as $\lat L, \lat M$ to denote lattices.  $n \in \Z_{>0}$ and $d \in \{1, \ldots, n-1\}$ are fixed integers throughout the paper.
 
 As in the statement of Theorem \ref{thm:thunder}, $\lambda_i(\lat L)$ is the $i$-th successive minimum of $\lat L$, and $\lat L_{i}$ denotes (a choice of) a primitive $i$-dimensional sublattice containing $v_1, \ldots, v_i \in \lat L$, which are linearly independent with $\|v_{i}\| = \lambda_i(\lat L)$. 
The $(i, j)$-entry of a matrix is denoted by the lowercase letter of the name of the matrix indexed by $ij$. For example, if $A$ is a $d \times n$ matrix, then $A = (a_{ij})_{1 \leq i \leq d \atop 1 \leq j \leq n}$. Similarly, if $x \in \R^n$, then the $i$-th entry of $x$ is denoted by $x_i$.

Later, given a $d \times (n-1)$ matrix $A$ and a $d \times 1$ vector $v$, we will need to consider the $d \times n$ matrix $B$ whose $i$-th row equals $(a_{i1}, \ldots, a_{i,n-1}, v_i)$. We write $B = (A ; v)$ to describe such a matrix.

For two quantities $f$ and $g$, $f \ll g$ means $f < Cg$, where $C$ is a positive constant possibly depending on $d$ and $n$ but no other variables. $f \sim g$ means $f \ll g$ and $g \ll f$. For example, Minkowski's second theorem says that $\det \lat L \sim \prod \lambda_i(\lat L)$.

For two matrices $A$ and $B$ with $d$ rows, $A \sim B$ means they differ by the left multiplication by an element of $\Gamma$. If the rows of each of $A$ and $B$ respectively span $\lat A$ and $\lat B$ in $\Gr(\lat M, d)$ (whose precise definition is given below), $A \sim B$ means that $\lat A = \lat B$.

Later in the paper, we will need a few facts from reduction theory. Let $\{v_1, \ldots, v_m\}$ be a basis of a lattice $\lat M$, and $\{v_1^*, \ldots, v_m^*\}$ be its Gram-Schmidt orthogonalization, that is, each $v_i^*$ is the projection of $v_i$ to the orthogonal complement of $\spn(v_1, \ldots, v_{i-1})$. Let us say the basis is \emph{reduced} if each $v_i \sim \lambda_i(\lat M)$ and $\langle v_i^*, v_j \rangle / \|v_i^*\|^2 \leq 1/2$ for all $i < j$. It is known by reduction theory (see e.g. \cite[Chapter 1]{Borel}) that any lattice has a reduced basis. Moreover, the LLL algorithm (\cite{LLL}) outputs a reduced basis of any lattice, given any basis of that lattice.

\subsubsection*{$\Gr(\lat M, d)$ and the determinant/height}
A $d \times n$ integral matrix $X \in \mathrm{Mat}_{d \times n}(\Z)$ is said to be \emph{primitive} if $X$ can be completed to an element of $\GL(n,\Z)$. When $d = 1$, this agrees with the standard definition of a primitive vector. We denote the set of all primitive $d \times n$ matrices by $\mathrm{Mat}^{pr}_{d \times n}(\Z)$.

For a lattice $\lat M \subseteq \R^n$ of rank $m \leq n$, a sublattice $\lat A \subseteq \lat M$ is said to be \emph{primitive} if $\spn_\R(\lat A) \cap \lat M = \lat A$. We denote $\Gr(\lat M, d)$ for the set of all rank $d$ primitive sublattices of $\lat M$ inside $\R^n$. Choose a basis $\{v_1, \ldots, v_m\}$ of $\lat M$, and a basis $\{w_1, \ldots, w_d\}$ of $\lat A \in \Gr(\lat M, d)$. Let $M$ and $A$ respectively denote the $m \times n$ and $d \times m$ matrices, such that the $i$-th row of $M$ is $v_i$, and the $i$-th row of $AM$ is $w_i$. One checks that $A \in \mathrm{Mat}^{pr}_{d \times m}(\Z)$ by the fact that $\lat A \in \Gr(\lat M, d)$.

Suppose one chooses a different basis $\{w'_1, \ldots, w'_d\}$ of $\lat A$, and let $A' \in \mathrm{Mat}^{pr}_{d \times m}(\Z)$ be the matrix such that the $i$-th row of $A'M$ is $w'_i$. Then $A = gA'$ for some $g \in \Gamma$. Conversely, if $A = gA'$ for some $g \in \Gamma$ and $A, A' \in \mathrm{Mat}^{pr}_{d \times m}(\Z)$, then the rows of $AM$ and $A'M$ span the same element of $\Gr(\lat M, d)$. Therefore, with a choice of $M \in \mathrm{Mat}_{m \times n}(\R)$ whose rows span $\lat M$ over $\Z$, there exists a bijection between $\Gr(\lat M, d)$ and the set of orbits $\Gamma \backslash (\mathrm{Mat}^{pr}_{d \times m}(\Z) \cdot M)$ of the action of $\Gamma$ on $\mathrm{Mat}^{pr}_{d \times m}(\Z) \cdot M$ by left multiplication.

To make this even more explicit, recall that each element of $\Gamma \backslash \mathrm{Mat}^{pr}_{d \times m}(\Z)$ is uniquely represented by a primitive Hermite normal form over $\Z$. Thus there is also a bijection between $\Gr(\lat M, d)$ and the set of all elements of form $AM$ where $A$ is a $d \times m$ primitive Hermite normal form over $\Z$. Whenever convenient, we will use these identifications of $\Gr(\lat M, d)$ interchangeably throughout the paper.

In order to simplify some notations, we adopt the unusual convention that all determinants are nonnegative (the groups $\GL(n, \cdot)$ and $\SL(n, \cdot)$ maintain their usual meanings, though). Specifically, for a square matrix $X$, we write $\det X$ for the absolute value of its usual definition of determinant. For a non-square matrix $X$, we define $\det X = \sqrt{\det XX^{\tr}}$. For a lattice $\lat A \subseteq \R^n$, we define $\det \lat A$ to be its covolume within its $\R$-span. For $\lat A \in \Gr(\lat M, d)$, note that $\det \lat A = \det AM$ holds, where $M$ is any choice of a matrix whose row vectors form a basis of $\lat M$, and $A$ is any choice of an element of $\mathrm{Mat}^{pr}_{d \times m}(\Z)$ such that the row vectors of $AM$ form a basis of $\lat A$.

For a matrix $M$, we define
\begin{equation*}
f_H(M) = \begin{cases} 1 &\mbox{if $\det M \leq H$} \\ 0 &\mbox{otherwise,} \end{cases}
\end{equation*}
and similarly for a lattice $\lat M$.

It is easy to see that, for any compactly supported function $F$ defined on a subset of $\R$, we have
\begin{equation*}
\sum_{\lat A \in \Gr(\lat M, d)} F(\det \lat A) = \sum_{\lat B \in \Gr(\Z^m, d)} F(\det \lat B M) = \sum_{B \in \Gamma \backslash \mathrm{Mat}^{pr}_{d \times m}(\Z)} F(\det BM),
\end{equation*}
where $\lat B M$ is understood as the image of $\lat B$ by the linear map $\R^n \rightarrow \R^n$ induced by the matrix $M$. Again we will switch freely between these notations as we see fit.

\subsubsection*{Orthogonality notions}
Following Schmidt (\cite{Sch}), we define the \emph{polar lattice} $\lat M^P$ of $\lat M$ by $\lat M^P = \{w \in \spn_\R(\lat M) :  \langle v, w \rangle \in \Z,  \forall v \in \lat M\}$. If $\lat S \in \Gr(\lat M, d)$, we define its \emph{orthogonal lattice} $\lat S^\perp \in \Gr(\lat M^P, m-d)$ by $\lat S^\perp = \{w \in \lat M^P : \langle v, w \rangle = 0, \forall v \in \lat S\}$.

In addition, for a $m \times n$ matrix $M$ whose $i$-th row vector is denoted by $v_i$, we define its \emph{polar matrix} $M^P$ as the $m \times n$ matrix whose $j$-th row vector $v^P_j$ lies in $\spn_\R(v_1, \ldots, v_m)$ and satisfies $\langle v_i, v^P_j \rangle = \delta_{ij}$. Then the rows of $M^P$ generate the polar lattice of the lattice generated by the rows of $M$.

\subsection{Base cases}

In case $d = 1$, Theorem \ref{thm:main} is precisely Theorem 4 in \cite{Thu2} (also Lemma 2 of \cite{Sch}), which states that
\begin{equation} \label{eq:case1}
P(\lat L, 1, H) = a(n,1) \frac{H^n}{\det \lat L} + O\left(\sum_{i=1}^n \frac{H^{n-i}}{\det \lat L_{n-i}}\right).
\end{equation}

Below in Lemma \ref{lemma:case1}, we present an extension of \eqref{eq:case1} to an affine lattice, which we will need later.

In case $d = n - 1$, we apply the \emph{duality theorem} (see Section 2 of \cite{Thu2}) to \eqref{eq:case1}, which states that, for a sublattice $\lat S \subseteq \lat L$ and its orthogonal lattice $\lat S^\perp \subseteq \lat L^P$,
\begin{equation*}
\det \lat S^\perp = \frac{\det \lat S}{\det \lat L}
\end{equation*}
holds, and thus
\begin{equation} \label{eq:duality}
P(\lat L, d, H) = P(\lat L^P, n-d,\frac{H}{\det \lat L}).
\end{equation}

Therefore \eqref{eq:case1} implies
\begin{equation*}
P(\lat L, n-1, H) = a(n, n-1) \frac{H^n}{(\det \lat L)^n \det \lat L^P} + O\left(\sum_{i=1}^n\frac{H^{n-i}}{(\det \lat L)^{n-i} \det (\lat L^P)_{n-i}}\right).
\end{equation*}

By the well-known facts that $\det \lat L \cdot \det \lat L^P = 1$ and $\lambda_i(\lat L) \lambda_{n-i+1}(\lat L^P) \geq 1$
(in fact, $\lambda_i(\lat L) \lambda_{n-i+1}(\lat L^P) \sim 1$, by \cite[Theorem 2.1]{Ban}),
\begin{equation} \label{eq:duality2}
\det (\lat L^P)_{n-i} \gg \det \lat L_{i} / \det \lat L,
\end{equation}
so we can rewrite the above as
\begin{equation*}
P(\lat L, n-1, H) = a(n, n-1) \cdot \frac{H^n}{(\det \lat L)^{n-1}} + O\left(\sum_{i=1}^n \frac{H^{n-i}}{\det \lat L_i \cdot (\det \lat L)^{n-1-i}}\right).
\end{equation*}


\section{Division into two parts}

\subsection{Preliminaries}
Until the end of Section 6, we fix $n \geq 4$ and $2 \leq d \leq n-2$. We will divide $P(\lat L, d, H)$ into two parts, and deal with them one at a time. We induct on $n$, assuming that $P$ has been computed for all lattices of rank $< n$.

Throughout the rest of the paper, we fix a basis $\{v_1, \ldots, v_n\}$ of $\lat L$, and denote by $L$ the $n \times n$ matrix whose $i$-th row is $v_i$. Define $\bar{\lat L} = \lat L /\langle v_n \rangle$, and identify it with the projection of $\lat L$ onto the subspace of $\R^n$ orthogonal to $v_n$ i.e. we think of $\bar{\lat L}$ as a subset of $\R^n$. Let $\bar{v}_i$ be the component of $v_i$ orthogonal to $v_n$, so that $v_i = \bar{v}_i + a_iv_n$ for some $a_i \in \R$ and $\bar{\lat L} = \spn_\Z(\bar{v}_1, \ldots, \bar{v}_{n-1})$. Also let $\bar L$ be the $(n-1) \times n$ matrix whose $i$-th row is $\bar v_i$.

We write
\begin{equation*}
P(\lat L, d, H) = P^1(\lat L, d, H) + P^2(\lat L, d, H),
\end{equation*}
where $P^1(\lat L, d, H)$ equals the number of rank $d$ primitive sublattices of $\lat L$ of determinant $\leq H$ such that its projection to $\bar{\lat L}$ is also of rank $d$, and $P^2(\lat L, d, H)$ equals the number of those whose projection is of rank $d-1$. Equivalently, $P^1$ counts primitive sublattices whose $\R$-span does not contain $v_n$, and $P^2$ counts those that does.

As discussed in Section 1.4 above, we may identify $\lat A \in \Gr(\lat L, d)$ with an orbit $\Gamma M L$ of the left multiplication of $\Gamma$ on $\mathrm{Mat}^{pr}_{d \times n}(\Z) \cdot L$, for some $M = (c_{ij})_{1 \leq i \leq d \atop 1 \leq j \leq n} \in \mathrm{Mat}^{pr}_{d \times n}(\Z)$. Also, let $\tilde{L}$ be the $n \times n$ matrix whose $i$-th row vector equals $\bar{v}_i$ for $1 \leq i \leq n-1$, and $v_n$ for $i = n$, so that
\begin{equation*}
L = \begin{pmatrix}
1 &    &           &    & a_1 \\
   & 1  &          &    & a_2 \\
   &    & \ddots &     & \vdots \\
   &    &           & 1 & a_{n-1} \\
   &    &           &    & 1
\end{pmatrix}\tilde{L}.
\end{equation*}

Then we can also write $\lat A$ in the form $\Gamma (C; c + c') \tilde{L}$, where $C = (c_{ij})_{1 \leq i \leq d \atop 1 \leq j \leq n-1}$ is the first $d \times (n-1)$ submatrix of $M$, and $c = (c_{1n}, \ldots, c_{dn})^{\mathrm{tr}}$ and $c' = (\sum_j a_jc_{1j}, \ldots, \sum_j a_jc_{dj})^{\mathrm{tr}}$ are vectors in $\R^d$.

\subsection{Computing $P^2(\lat L, d, H)$}
Consider first the case $\rank\, C = d - 1$, so that $\lat A$ contributes to $P^2$. We may assume that $M$ is a Hermite normal form, so that $C$ is too. Because $M$ is primitive, so is $C$, and the $d$-th entry of the vectors $c$ and $c'$ must be equal to $1$ and $0$ respectively. This forces each of the other entries of $c + c'$ to have only one choice modulo the left action of $\Gamma$. Thus
\begin{equation} \label{eq:cau-bin}
P^2(\lat L, d, H) = P(\bar{\lat L}, d-1, \frac{H}{\|v_n\|}),
\end{equation}
to which we can simply apply Theorem \ref{thm:main} (see the remark after its statement). 

\subsection{Some lemmas}
Working with $P^1$ is much more involved. Most of the remainder of this paper is devoted to this task. The goal of this section is to derive the expression \eqref{eq:p1mid} for $P^1$ that is amenable to computation.

We start by recalling the standard choice of the representatives of the right cosets of $\Gamma$ in the double coset $\Gamma a \Gamma$, where $a \in \mathrm{Mat}_{d \times d}(\Z)$ has determinant $k > 0$. Such a representative, say $h = (h_{ij})_{1 \leq i, j \leq d}$, is a lower triangular matrix with determinant $k$, with the condition that $0 \leq h_{ji} < h_{ii}$ for all $j > i$. Of course, $\Gamma h \subseteq \Gamma a \Gamma$ if and only if $a$ and $h$ have the same invariant factors.

\begin{lemma}\label{lemma:triple}
Given an integral $d \times n$ matrix $(C; c)$ with $\rank\, C = d$, there exists a unique triple $(h, B, b)$, where $h$ is one of the right coset representatives described above, $B$ is a $d \times (n-1)$ primitive Hermite normal form of rank $d$, and $b \in \Z^n$, such that $(C; c) \sim (hB; b)$.
\end{lemma}
\begin{proof}
By the theory of the Smith normal form, we have $(C; c) \sim (aB_0; b_0)$ where $a$ is an invariant factor matrix --- that is, $a = \mathrm{diag}(a_1, \ldots, a_d)$ with $a_i | a_{i+1}$ --- $B_0$ is a primitive $d \times (n-1)$ matrix of full rank, and $b_0 \in \Z^d$. Write $B_0 = \gamma B$, where $B$ is the Hermite normal form of $B_0$ and $\gamma \in \Gamma$. Then there exists $\gamma' \in \Gamma$ and $h$ a coset representative of $\Gamma a \Gamma$ such that $\gamma'h = a\gamma$. Therefore, writing $b = \gamma'^{-1}b_0$, we have $(C; c) \sim (hB, b)$.

Suppose we have another triple $(h', B', b')$ such that $(hB, b) \sim (h'B' , b')$. This is possible only if the row vectors of $B$ and $B'$ generate the same lattice. Since both $B$ and $B'$ are in the Hermite normal form, $B = B'$. This in turn implies $h = h'$ and $b = b'$.
\end{proof}

\begin{lemma}\label{lemma:prim}
Again given an integral $d \times n$ matrix $(C; c)$, write $C = \gamma a B$, where $\gamma \in \Gamma$, $a = \mathrm{diag}(a_1, \ldots, a_d)$ is an invariant factor matrix, and $B$ is primitive.
Thus $(C; c) \sim (aB; \gamma^{-1}c) = (aB; b)$, where $b := \gamma^{-1}c$.

Then $(aB; b)$ is primitive if and only if $a_1 = \ldots = a_{d-1} = 1$ and $b_d$ is coprime to $a_d$.
\end{lemma}
\begin{proof}
Without loss of generality, we may assume $B$ to be the matrix which has $1$'s in the diagonal and $0$'s elsewhere. $(aB, b)$ is imprimitive if and only if there exist integers $0 \leq r_i < a_i$ for $i = 1, \ldots, d$, $r_i$ not all zero, such that $(r_1, \ldots, r_d, 0, \ldots, 0, \sum_i b_ir_i/a_i) \in \Z^n$, or equivalently $\sum_i b_ir_i/a_i \in \Z$.

Suppose $a_{d-1} \neq 1$. We claim that, for any $b_{d-1}$ and $b_d$, $b_{d-1}r_{d-1}/a_{d-1} + b_dr_d/a_d \in \Z$ for a nontrivial choice of the $r$'s. There exists a prime $p$ such that $p | a_{d-1}$ and $p | a_d$, so it suffices to find a nontrivial solution to the expression $b_{d-1}r_{d-1} + b_dr_{d} \equiv 0 (\mathrm{mod}\, p)$. But this is clearly possible.

Next suppose $a_{d-1} = 1$. We are led to consider the condition $b_dr_d/a_d \in \Z$. This is impossible if and only if $(b_d, a_d) = 1$, which completes the proof.
\end{proof}

\begin{lemma}\label{lemma:countsub}
Write $e(p^\alpha) = \mathrm{diag}(1, \ldots, 1, p^\alpha)$. Then the necessary and sufficient condition for $h \in \mathrm{Mat}_{d \times d}(\Z)$ to be one of the standard form right coset representatives of $\Gamma$ in $\Gamma e(p^\alpha) \Gamma$ is as follows: $h$ is a lower triangular matrix with $h_{ii} = p^{a_i}$, where $a_i \geq 0$ and $\sum a_i = \alpha$, $0 \leq h_{ji} < h_{ii}$ for $j > i$, and in addition if $i < j$ are two indices such that $a_i, a_j \geq 1$ and $a_{i+1} = \ldots = a_{j-1} = 0$ --- i.e. all diagonal entries between $h_{ii}$ and $h_{jj}$ are trivial --- then $(h_{ji}, p) = 1$.
\end{lemma}
\begin{proof}
Let $h$ be a coset representative of some double coset of a matrix of determinant $p^\alpha$, in the form that we chose in the beginning of this section. Then all but the last condition are automatically satisfied. For the last condition, choose the three smallest indices $i < j < k$ for which $a_i, a_j, a_k > 0$. We consider the $3 \times 3$ matrix
\begin{equation} \label{eq:lemma:countsub}
\begin{pmatrix}
p^{a_i} &  &  \\
h_{ji} & p^{a_j} & \\
h_{ki} & h_{kj} & p^{a_k}
\end{pmatrix}.
\end{equation}

We will show that this matrix has invariant factors $(1, 1, p^{a_i + a_j + a_k})$ if and only if $h_{ji}$ and $h_{kj}$ are coprime to $p$. Then the proof is complete because we can repeatedly apply this argument to $h$ to compute the invariant factors of $h$.

If $h_{ji}$ and $p$ are coprime, there exist integers $x, y$ such that $yh_{ji} - xp^{a_i} = 1$, so that the matrix
\begin{equation*}
\begin{pmatrix}
h_{ji} & p^{a_{i}} & 0 \\
x & y & 0 \\
0 & 0 & 1
\end{pmatrix}
\end{equation*}
has determinant $1$. Multiplying this on the left of \eqref{eq:lemma:countsub}, we have
\begin{equation*}
\begin{pmatrix}
0 & p^{a_i + a_j} & 0 \\
1 & yp^{a_j} & 0 \\
h_{ki} & h_{kj} & p^{a_k}
\end{pmatrix},
\end{equation*}
which, upon multiplying by suitable elements of $\Gamma$ from both sides, becomes
\begin{equation*}
\begin{pmatrix}
1 & 0 & 0 \\
0 & p^{a_i + a_j} & 0 \\
0 & h_{kj} - yp^{a_j}h_{ki} & p^{a_k}
\end{pmatrix}.
\end{equation*}

If furthermore $h_{kj}$ is coprime to $p$, then so is $h_{kj} - yp^{a_j}h_{ki}$, so we can use the same trick to see that \eqref{eq:lemma:countsub} has invariant factors $(1, 1, p^{a_i+a_j+a_k})$ indeed.

Now go back to \eqref{eq:lemma:countsub} and consider the case $h_{ji} = cp^b$; we can assume $1 \leq b < a_j$ and $(c, p) = 1$. We restrict our attention to the $2 \times 2$ upper-left corner submatrix of \eqref{eq:lemma:countsub}, and temporarily use $\approx$ to denote the equivalence under the left and right multiplication by $\Gamma$. Then, by a similar argument as earlier, for an appropriate integer $y$,
\begin{equation*}
\begin{pmatrix}
p^{a_i} & \\
cp^b & p^{a_j}
\end{pmatrix}
=
\begin{pmatrix}
p^{a_i - b} & \\
c & p^{a_j}
\end{pmatrix}
\begin{pmatrix}
p^b & \\
 & 1
\end{pmatrix}
\approx
\begin{pmatrix}
0 & p^{a_i + a_j - b} \\
1 & yp^{a_j}
\end{pmatrix}
\begin{pmatrix}
p^b & \\
 & 1
\end{pmatrix}
\approx
\begin{pmatrix}
0 & p^{a_i + a_j - b} \\
p^b & 0
\end{pmatrix},
\end{equation*}
so $p^b$ appears as one of the invariant factors.

\end{proof}

\begin{lemma}\label{lemma:count}
Write $e(k) = \mathrm{diag}(1, \ldots, 1, k)$, as in the previous lemma. Then the number of the right cosets of $\Gamma$ in $\Gamma e(k) \Gamma$ equals
\begin{equation*}
\prod_{p | k \atop p^\alpha \| k} p^{(\alpha-1)(d-1)}(1 + p + \ldots + p^{d-1}).
\end{equation*}
\end{lemma}
\begin{proof}
From the general theory of Hecke operators (see Chapter 3 of Shimura \cite{Shi}), it suffices to prove the lemma for the case $k = p^\alpha$. We proceed by induction on $\alpha$.

In case $\alpha = 1$, there exist $p^{d-i}$ coset representatives which has $a_{ii} = p$ and $a_{jj} = 1$ for all $j \neq i$. This exhausts all the representatives of $\Gamma e(p) \Gamma$, so the lemma holds true in this case.

For the general case, it suffices to match, to each representative $h$ of $\Gamma e(p^{\alpha-1}) \Gamma$, $p^{d-1}$ representatives of $\Gamma e(p^{\alpha}) \Gamma$, different for each $h$. Suppose $j$ is the smallest number for which $h_{jj}$ is a power of $p$. Then modifying $h_{jj}$ to $ph_{jj}$ and $h_{kj} (k > j)$ to $h_{kj} + c_kh_{jj}$, for any choice of $0 \leq c_k < p$, yields a representative of $\Gamma e(p^{\alpha}) \Gamma$, accounting for $p^{d-j}$ out of $p^{d-1}$ total. Also, for each $i < j$, replacing $h_{ii} (=1)$ by $p$, a choice of each $h_{ki}$ $(k \neq j)$ from $\{0, \ldots, p-1\}$ and of $h_{ji}$ from $\{1, \ldots, p-1\}$ ($h_{ji}$ cannot be $0$ by the previous lemma) yields a representative of $\Gamma e(p^\alpha)\Gamma$, and there are $p^{d-i-1}(p-1)$ of this kind. Therefore, for each $h$ there is a total of $p^{d-j} + p^{d-j}(p-1) + p^{d-j+1}(p-1) + \ldots + p^{d-2}(p-1) = p^{d-1}$ coset representatives of $\Gamma e(p^{\alpha}) \Gamma$ constructed in this manner, as desired. It remains to show that these representatives do not overlap with those constructed from a different choice of $h$. But this is immediate since, given a representative of $\Gamma e(p^{\alpha}) \Gamma$, one can read off which representative of $\Gamma e(p^{\alpha-1}) \Gamma$ it came from, by discarding the first factor of $p$ that appears in its diagonal.
\end{proof}

\subsection{A computable expression for $P^1(\lat L,d,H)$}

For $\lat A \in \Gr(\lat L, d)$, recall we defined $f_H(\lat A) = 1$ if $\det \lat A \leq H$ and $0$ otherwise. Also, as in the statement of Lemma \ref{lemma:count} write $e(k) := \mathrm{diag}(1, \ldots, 1, k)$. Thanks to Lemmas \ref{lemma:triple}, \ref{lemma:prim} and \ref{lemma:count}, we can rewrite $P^1(\lat L, d, H)$ as
\begin{equation} \label{eq:p1rewrite}
\sum_{B \in \Gamma \backslash \mathrm{Mat}^{pr}_{d \times (n-1)}(\Z)} \sum_{k \geq 1} \sum_{h} \sum_{b \in \Z^d \atop (hB; b)\, \mathrm{prim.}} f_H\left(
\left(hB; b \right) L \right),
\end{equation}
where the sum over $h$ is taken over all coset representatives of $\Gamma e(k)\Gamma$ in the standard form.


Fix $h, k, B$ for a moment, and consider the innermost summation in \eqref{eq:p1rewrite}. For some $B' \sim B$, it is equal to (\emph{cf}. Lemma \ref{lemma:prim})

\begin{align} \notag
& \sum_{b \in \Z^d \atop (k, b_d) = 1} f_H\left(\left(e(k)B'; b \right) L\right) \\ \notag
&= \sum_{l | k}\mu(l)\sum_{b \in \Z^d} f_H\left(\left(e(k)B'; e(l)b \right) L\right) \\ \notag
&= \sum_{l | k}\mu(l)\sum_{b \in \Z^d} f_H\left(\left(e(k)B'; e(l)b + e(k)t \right) \tilde{L}\right) \\ 
&= \sum_{l | k}\mu(l)\sum_{b \in \Z^d} f_H\left(e(k)B'\bar{L} + (e(l)b + e(k)t)v_n \right), \label{eq:p1inner}
\end{align}
where $\mu$ is the M\"obius function, and we wrote
\begin{equation*}
t = \begin{pmatrix} \sum_j a_jb'_{1j} \\ \vdots \\ \sum_j a_jb'_{dj} \end{pmatrix}
\end{equation*}
for short. Note that $v_n$ is a row vector, whereas $b$ and $t$ are column vectors.

Temporarily write $X = (e(l)b + e(k)t)v_n$ and $Y = e(k)B'\bar{L}$. We need to compute the determinant of $X + Y$. First observe that
\begin{equation*}
(X+Y)(X+Y)^\tr = XX^\tr + XY^\tr + YX^\tr + YY^\tr = XX^\tr + YY^\tr,
\end{equation*}
because $v_n \bar L^\tr = 0$, and also
\begin{equation*}
XX^\tr = \|v_n\|^2 \left(e(l)b + e(k)t\right)\left(e(l)b + e(k)t\right)^\tr.
\end{equation*}
This motivates the use of the matrix-determinant lemma, which asserts that for a $d \times d$ matrix $A$ and (row) vectors $x, y \in \R^d$, $\det (A + x^\tr y) = \det A \cdot (1 + y A^{-1} x^\tr)$. To this end, we also need the following lemma.

\begin{lemma}
Let $Y$ be a full-rank $d \times n$ matrix whose $i$-th row equals $y_i \in \mathbb{R}^n$. Let $z_1, \ldots, z_d \in \mathbb{R}^n$ such that they form the basis of the polar lattice spanned by $y_1, \ldots, y_d$ and that $\langle z_i, y_j \rangle = \delta_{ij}$. Let $Z = Y^P$ be the $d \times n$ matrix whose $i$-th row equals $z_i$. Then the inverse of $YY^\tr$ is given by $ZZ^\tr$.
\end{lemma}
\begin{proof}
Complete $Y$ to an invertible $n \times n$ matrix $\bar{Y} = \binom{Y}{Y'}$, such that the rows of $Y'$ are orthogonal to the rows of $Y$. Similarly complete $Z$ to $\bar{Z} = \binom{Z}{Z'}$, so that the rows of $\bar{Z}$ form the dual basis to that formed by the rows of $\bar{Y}$. Then the rows of $Z'$ are orthogonal to the rows of $Z$ as well.

Since $\bar{Z}$ and $\bar{Y}^\tr$ are inverses of each other, we have $\bar{Y}\bar{Y}^\tr\bar{Z}\bar{Z}^\tr = I$. By abuse of language, write $Y = \binom{Y}{0}, Y' = \binom{0}{Y'}$, and similarly with $Z$. Then
\begin{equation*}
\bar{Y}\bar{Y}^\tr\bar{Z}\bar{Z}^\tr = (Y + Y')(Y^\tr Z + Y'^\tr Z')(Z^\tr + Z'^\tr) = YY^\tr ZZ^\tr + Y'Y'^\tr Z'Z'^\tr,
\end{equation*}
and observe that the first term on the right is zero outside the first $d \times d$ submatrix, and the second term is zero outside the ``last'' $(n-d) \times (n-d)$ submatrix. This completes the proof.
\end{proof}

Thanks to the above lemma, with $Z = Y^P$ we compute that $\det (X+Y)$ is the square root of
\begin{align*}
\det(XX^\tr + YY^\tr) &= \det\left(\|v_n\|^2\left(e(l)b + e(k)t\right)\left(e(l)b + e(k)t\right)^\tr + YY^\tr\right) \\
&= \det(YY^\tr)\left(1+ \|v_n\|^2 \left(e(l)b + e(k)t\right)^\tr (ZZ^\tr) \left(e(l)b + e(k)t\right) \right) \\
&= k^2\det(B'\bar{L})^2\left(1 + \|v_n\|^2\left\|(e(l)b + e(k)t)^\tr e(k^{-1})(B'\bar{L})^P\right\|^2\right).
\end{align*}
In the last line, we used the fact that $Z = \left( e(k)B'\bar{L} \right)^P = e(k^{-1})(B'\bar{L})^P$.

Let
\begin{equation*}
K(B) = \frac{1}{\|v_n\|}\sqrt{\frac{H^2}{k^2\det(B\bar{L})^2} - 1}
\end{equation*}
if $H \geq k\det(B\bar{L})$, and set $K(B) = 0$ otherwise. We will use this notation throughout the rest of the paper. Then \eqref{eq:p1inner} becomes

\begin{equation*}
\sum_{l | k}\mu(l) \cdot\left(
\mbox{\begin{tabular}c
number of vectors (nonzero, if $k \neq 1$) in the lattice \\
spanned by the rows of $e(l/k)(B'\bar{L})^P$\\
whose translates by $t^\tr(B'\bar L)^P$ have length $\leq K(B')$
\end{tabular} }\right).
\end{equation*}

The lemma below ensures that the translation of the vectors by $t^\tr(B'\bar L)^P$ does not present any extra difficulty in our estimate of this sum.

\begin{lemma} \label{lemma:case1}
Let $\Lambda \in \R^d$ be a lattice of rank $d$, and $t \in \R^d$. Temporarily denote by $N(r)$ the number of points $v \in \Lambda + t$ with $\|v\| \leq r$. Then
\begin{equation*}
N(r) = \frac{V(d)r^d}{\det \Lambda} + O\left(\sum_{i=1}^d \frac{r^{d-i}}{\det \Lambda_{d-i}}\right),
\end{equation*}
where the implicit constant depends on $d$ only.
\end{lemma}
\begin{proof}
This is Lemma 2 in \cite{Sch} generalized to an affine lattice, and is also a special case of Theorem 5.4 in \cite{Wid}. We provide a proof here for completeness.

We proceed by induction on $d$. The base case $d=1$ is clear. Now assume the lemma for $d-1$. By adjusting $\det \Lambda$, we may assume $r = 1$.

First consider the case $\lambda_d \leq 1$. Let $x_i \in \Lambda$, $i \in \{1, \ldots, d\}$, be a vector with $\|x_i\| = \lambda_i$, and consider the parallelepiped spanned by $x_1, \ldots, x_d$. Its diameter is $\leq \lambda_1 + \ldots + \lambda_d \leq d\lambda_d$, and it contains a fundamental parallelepiped $F$ of $\Lambda$, which also has diameter $\leq d\lambda_d$.

Write $B(s)$ for the ball in $\R^n$ at the origin of radius $s$. Then since $B(\max(0,1-d\lambda_d)) \subseteq (\Lambda + t) \cap B(1) + F \subseteq B(1 + d\lambda_d)$, we have
\begin{align*}
|N(r)\det\Lambda - V(d)| &\leq V(d)((1+d\lambda_d)^d - \max(0,1-d\lambda_d)^d) \\
&\leq V(d)(2d\lambda_d)^dd,
\end{align*}
and thus
\begin{equation*}
\left|N(r) - \frac{V(d)}{\det\Lambda}\right| = O\left(\frac{\lambda_d}{\det \Lambda}\right) = O\left(\frac{1}{\det \Lambda_{d-1}}\right),
\end{equation*}
where the second equality follows from the Minkowski's second theorem.

It remains to consider the case $\lambda_d > 1$. Then $(\Lambda + t) \cap B(1)$ lies in at most two translates of $\Lambda_{d-1}$ in the direction of $x_d$. Thus the induction hypothesis implies $N(r) = O\left(\sum_{i=1}^d 1/\det \Lambda_{d-i}\right)$. Also we have
\begin{equation*}
\frac{1}{\det\Lambda} < \frac{\lambda_d}{\det\Lambda} = O\left(\frac{1}{\det\Lambda_{d-1}}\right)
\end{equation*}
as above. This completes the proof.

\end{proof}

It follows that \eqref{eq:p1inner} equals
\begin{equation*}
\sum_{l | k} \mu(l) \left(\frac{V(d)K(B')^d}{\det \mathfrak L(l/k, (B'\bar{L})^P)} + O\left(\sum_{i=1}^d \frac{K(B')^{d-i}}{\det\mathfrak L(l/k, (B'\bar{L})^P)_{d-i}}\right)\right),
\end{equation*}
where $\mathfrak L(x, M)$ here denotes the lattice spanned by the row vectors of $e(x)M$. We have $\mathfrak L(l/k, (B'\bar L)^P) = \mathfrak L(k/l, B'\bar L)^P$, and $\det (\mathfrak L(k/l, B'\bar L)^P)_{d-i} \gg \det \mathfrak L(k/l, B'\bar L)_{i} / \det\mathfrak L(k/l, B'\bar L)$ by \eqref{eq:duality2}. Also, $\det \mathfrak L(k/l, B'\bar L)_{i} \gg \det \mathfrak L(1, B'\bar L)_{i}$, so the above sum can be rewritten as
\begin{equation*}
\sum_{l | k} \mu(l) \frac{k}{l}\left(\frac{V(d)K(B')^d}{\det \mathfrak L(1, B'\bar L)^P} + O\left(\sum_{i=1}^d \frac{K(B')^{d-i}\det \mathfrak L(1, B'\bar L)}{\det\mathfrak L(1, B'\bar L)_i}\right)\right),
\end{equation*}
which we in turn rewrite as, for the lattice $\lat B \in \Gr(\Z^{n-1},d)$ spanned by $B'$,
\begin{equation*}
\sum_{l | k} \mu(l) \frac{k}{l}\left(\frac{V(d)K(\lat B)^d}{\det (\lat B \bar L)^P} + O\left(\sum_{i=1}^d \frac{K(\lat B)^{d-i}\det \lat B \bar L}{\det (\lat B\bar L)_i}\right)\right).
\end{equation*}

Summing up all our work in this section, we deduce that \eqref{eq:p1rewrite} equals
\begin{align}
\notag& \sum_{\lat B \in \Gr(\Z^{n-1}, d)} \sum_{k \geq 1} \prod_{p | k \atop p^\alpha \| k} p^{(\alpha-1)(d-1)}(1 + p + \ldots + p^{d-1}) \sum_{l | k} \mu(l) \frac{k}{l}\left(\frac{V(d)K(\lat B)^d}{\det (\lat B \bar L)^P} + O\left(\sum_{i=1}^d \frac{K(\lat B)^{d-i}\det \lat B \bar L}{\det (\lat B\bar L)_i}\right)\right)\\
\label{eq:p1mid}&= \sum_{k \geq 1} \prod_{p | k \atop p^\alpha \| k} p^{(\alpha-1)(d-1)}(1 + p + \ldots + p^{d-1}) \varphi(k) V(d)\sum_{\lat B \in \Gr(\Z^{n-1}, d)} \left(K(\lat B)^d\det \lat B \bar L + O\left(\sum_{i=1}^d \frac{K(\lat B)^{d-i}\det \lat B \bar L}{\det (\lat B\bar L)_i}\right)\right). \\ \notag
\end{align}
Here $\varphi(k) = \sum_{l | k} \mu(l)\frac{k}{l}$ is the Euler totient.

The remainder of this paper is devoted to computing \eqref{eq:p1mid}. Because $K(\lat B)$ depends on $k$, we cannot deal with the constant factor just yet. However, we will later use

\begin{lemma}\label{lemma:zeta}
For $m > d + 1$,
\begin{equation*}
\sum_{k \geq 1} \prod_{p | k \atop p^\alpha \| k} p^{(\alpha-1)(d-1)}(1 + p + \ldots + p^{d-1}) \cdot \varphi(k)k^{-m} = \frac{\zeta(m-d)}{\zeta(m)}.
\end{equation*}
\end{lemma}
\begin{proof}
We can write the expression under question multiplicatively as
\begin{equation*}
\sum_{k \geq 1} \prod_{p | k \atop p^\alpha \| k} p^{-(m-d)\alpha}\left(1 - \frac{1}{p^d}\right) = \prod_p \left(1 + \sum_{i \geq 1}(1 - p^{-d})p^{-i(m-d)}\right),
\end{equation*}
which then becomes
\begin{align*}
&\prod_p \left(\sum_{i \geq 0} p^{-i(m-d)} - p^{-m}\sum_{i \geq 0} p^{-i(m-d)}\right) \\
&=\prod_p (1 - p^{-m})(1 - p^{m-d})^{-1} \\
&=\frac{\zeta(m-d)}{\zeta(m)}.
\end{align*}
\end{proof}


\section{Main term of \eqref{eq:p1mid}}

In this section, we estimate the intended main term of \eqref{eq:p1mid}, namely
\begin{equation} \label{eq:maingoal}
\sum_{\lat B \in \Gr(\Z^{n-1}, d)} K(\lat B)^d\det(\lat B\bar{L}),
\end{equation}
for each $k \geq 1$ and $2 \leq d \leq n-2$. We may also assume $H \geq k\min_\lat B \det(\lat B\bar{L})$, since otherwise \eqref{eq:maingoal} is equal to $0$. Our approach is essentially that of Schmidt \cite{Sch}, who uses summation by parts. We improve it somewhat by adopting the language of the Riemann-Stieltjes integral, in order to simplify the computation and to derive pretty error terms.

Let us rewrite \eqref{eq:maingoal} as
\begin{equation} \label{eq:4.1re}
\frac{1}{\|v_n\|^dk^d} Q(k,H),
\end{equation}
where
\begin{equation*}
Q(k, H) := \sum_{\lat B \in \Gr(\Z^{n-1}, d)} \psi(\det(\lat B\bar{L}))
\end{equation*}
and 
\begin{equation*}
\psi(t) = \begin{cases} t((H/t)^2 - k^2)^{d/2} & \mbox{for $0 < t \leq H/k$} \\ 0 & \mbox{otherwise} \end{cases}.
\end{equation*}
It is easy to check that $\psi(t)$ is a twice differentiable function on $0 < t \leq H/k$, with $\psi'(t) = -((d-1)(H/t)^2 + k^2)((H/t)^2 - k^2)^{(d/2-1)} \leq 0$.

Choose a $\delta > 0$ with $\delta \leq \min_\lat B \det(\lat B\bar{L})$. Write $H/k = (\alpha + s)\delta$ with $\alpha \in [0, 1)$ and $s \in \Z$. Also, let $P_1(t)$ be the number of elements $\lat B \in \Gr(\Z^{n-1}, d)$ such that $t < \det(\lat B\bar{L}) \leq t + \delta$ , and $P_2(t) = P_1(t-\delta)$. Then for $i = 1, 2$,

\begin{equation*}
(-1)^i\left(Q(k, H) - \sum_{j=0}^{s-1} \psi((\alpha+j)\delta)P_i((\alpha+j)\delta)\right) \geq 0.
\end{equation*}

Write $R_1(t) = P(\bar{\lat L}, d, t+\delta)$ and $R_2(t) = P(\bar{\lat L}, d, t)$. Since $\psi((a + s)\delta) = 0$, by the summation by parts,
\begin{align*}
&(-1)^i\Big(Q(k,H) - \sum_{j=0}^{s-1} R_i((\alpha+j)\delta)(\psi((\alpha+j)\delta) - \psi((\alpha+j+1)\delta))\Big) \geq 0.
\end{align*}

Thus we have bounded $Q(k, H)$ from both sides by certain Riemann-Stieltjes sums. We need to show that those sums  converge as $\delta \rightarrow 0$ (and thus $s \sim H/k\delta \rightarrow \infty$). First, observe that, since $R_i$'s are supported strictly away from zero by $\varepsilon = \min_\lat B \det(\lat B\bar{L})$, we may assume the same of $\psi$, so that $\psi$ is of bounded variation. Second, $R_i$ are clearly not continuous, but by the induction hypothesis on $n$, we know it is bounded from both sides by a polynomial in $t$. More precisely,
\begin{equation*}
R_2(t) = a(n-1,d)\frac{t^{n-1}}{\det(\bar{\lat L})^d} + O\left(\sum_{j\in E_{n-1,d}} c_j t^{\gamma_j}\right)
\end{equation*}
where $c_j = b_j(\bar{\lat L})$ is as in Theorem \ref{thm:main}, and
\begin{equation*}
R_1(t) = R_2(t) + O_{t, \bar{\lat L}}(\delta).
\end{equation*}
By Theorem 6.8 of Rudin (\cite{Rudin}), we have shown that
\begin{equation} \label{eq:oink}
Q(k, H) = \frac{a(n-1,d)}{(\det\bar{\lat L})^d}\int_\varepsilon^{H/k} -t^{n-1}\psi'(t)dt + O\left(\sum_{j \in E_{n-1,d}} c_j \int_\varepsilon^{H/k} -t^{\gamma_j}\psi'(t)dt \right).
\end{equation}

Since the same argument will be used repeatedly later in this paper, we summarize our discussion so far in the form of a lemma:
\begin{lemma} \label{lemma:r-s}
Assume Theorem \ref{thm:main} for $n=m$, and let $\lat M$ be a lattice of rank $m$. Suppose $\psi$ is a decreasing twice differentiable function supported on $[a, b]$. Then
\begin{equation*}
\sum_{\lat B \in \Gr(\lat M, d)} \psi(\det \lat B) = \frac{a(m,d)}{(\det \lat M)^d} \int_a^b -t^m\psi'(t) dt + O\left(\sum_{j \in E_{m,d}} b_j(\lat M) \int_a^b -t^{\gamma_j} \psi'(t)dt \right).
\end{equation*}
\end{lemma}

We return to estimating \eqref{eq:oink}. Recall $\varepsilon = \min_\lat B \det(\lat B\bar{L}) \sim \prod_{i=1}^d \lambda_i(\bar{\lat L})$. In \eqref{eq:oink}, for the integrals inside the $O$-notation, there is no harm in replacing $\varepsilon$ with $0$ if $\gamma > d-1$. For the main term, we can do the same at the cost of
\begin{equation*}
\frac{1}{(\det\bar{\lat L})^d}\int_0^\varepsilon -t^{n-1}\psi'(t) dt \ll \frac{1}{(\det\bar{\lat L})^d}\int_0^\varepsilon H^dt^{n-d-1} dt \sim \frac{H^d\varepsilon^{n-d}}{(\det\bar{\lat L})^d} \ll \frac{H^d}{\varepsilon^{d-1}}.
\end{equation*}

Now the main term of $Q(k,H)$ contributes

\begin{align*}
&\frac{a(n-1,d)}{(\det\bar{\lat L})^d}\int_0^{H/k} -t^{n-1}\psi'(t)dt \\
&= \frac{a(n-1,d)}{(\det\bar{\lat L})^d}\left(-t^{n-1}\psi(t)\Big|_0^{H/k} + (n-1)\int_0^{H/k} t^{n-2}\psi(t)dt\right) \\
&= \frac{a(n-1,d)}{(\det\bar{\lat L})^d}(n-1)\int_0^{H/k} t^{n-1}\left(\frac{H^2}{t^2} - k^2\right)^{\frac{d}{2}}dt \\
&= \frac{a(n-1,d)}{(\det\bar{\lat L})^d}(n-1)H^nk^{-n+d}\int_0^1x^{n-d-1}(1-x^2)^{\frac{d}{2}}dt \\
&= \frac{a(n-1,d)}{(\det\bar{\lat L})^d}\frac{(n-1)V(n)}{(n-d)V(n-d)V(d)}H^nk^{-n+d} \\
&= \frac{a(n,d)}{(\det\bar{\lat L})^d}\frac{\zeta(n)}{\zeta(n-d)V(d)}H^nk^{-n+d}.
\end{align*}
For the second last equality, we used the identity on the beta function (see e.g. \cite[Section 6.2.1]{abs})
\begin{equation*}
\mathrm{B}(a, b) = 2\int_0^1 x^{2a-1}(1-x^2)^{b-1}dx,
\end{equation*}
and the last equality follows from the definition of $a(n,d)$.

Now accounting for the factor of $1/(\|v_n\|^dk^d)$ in $\eqref{eq:4.1re}$, we obtain for the intended main term of \eqref{eq:maingoal}
\begin{equation*}
\frac{a(n,d)}{(\det\bar{\lat L})^d\|v_n\|^dk^d}\frac{\zeta(n)}{\zeta(n-d)V(d)}H^nk^{-n+d} = \frac{a(n,d)}{(\det{\lat L})^d}\frac{\zeta(n)}{\zeta(n-d)V(d)}H^nk^{-n}.
\end{equation*}
It is clear that this term is scale-invariant i.e. invariant under replacing $\lat L$ and $H$ by $c\lat L$ and $c^dH$ for any $c>0$.



The error terms of $Q(k,H)$ are dealt with in a similar way, only simpler. For $j \in E_{m,d}$ with $\gamma_j > d-1$, the corresponding term contributes
\begin{align*}
&c_j\int_{\varepsilon}^{H/k} -t^{\gamma_j}\psi'(t)dt \\
&\leq c_j\int_{0}^{H/k} -t^{\gamma_j}\psi'(t)dt \\
&= -c_j t^{\gamma_j}\psi(t)\Big|_0^{H/k} + c_j \gamma_j\int_0^{H/k}t^{\gamma_j-1}\psi(t)dt \\
&\ll c_j \int_0^{H/k} t^{\gamma_j} (H/t)^d dt \\
&\ll c_j H^{\gamma_j+1}k^{d-\gamma_j-1}.
\end{align*}
It is apparent that $c_j H^{\gamma_j+1}k^{d-\gamma_j-1} / (\|v_n\|^dk^d)$ is scale-invariant, since both $c_jH^{\gamma_j}$ and $H/\|v_n\|^d$ are.

For those with $\gamma_j \leq d-1$, we proceed as follows:
\begin{align*}
&c_j\int_{\varepsilon}^{H/k} -t^{\gamma_j}\psi'(t)dt \\
&= -c_j t^{\gamma_j}\psi(t)\Big|_\varepsilon^{H/k} + c_j \gamma_j\int_\varepsilon^{H/k}t^{\gamma_j-1}\psi(t)dt \\
&\leq c_j\varepsilon^{\gamma_j+1-d}H^d + c_j\gamma_j\int_\varepsilon^{H/k}t^{\gamma_j}(H/t)^d dt \\
&\ll \begin{cases} c_jH^d\varepsilon^{-d+\gamma_j+1} &\mbox{if $\gamma_j < d-1$} \\ c_j H^{d}(1+\log\frac{H}{k\varepsilon} )&\mbox{if $\gamma_j = d-1$} \end{cases}.
\end{align*}
In case $\gamma_j = d-1$, we used our assumption $H \geq k\varepsilon$. Also, to retain the polynomial shape of the error term, we note $c_j H^{d}(1+\log\frac{H}{k\varepsilon}) = O\left(c_j H^{d+\eta}(k\varepsilon)^{-\eta}\right)$ for any $\eta > 0$, and use this bound instead. The scale-invariance can be checked in a straightforward manner.

In conclusion, we proved that \eqref{eq:maingoal} equals
\begin{equation*}
\frac{a(n,d)}{\det \lat L^d}\frac{\zeta(n)}{\zeta(n-d)V(d)}(H/k)^n + O\left(\sum_{j \in E^{(1)}} c'_j(H/k)^{\gamma_j}\right),
\end{equation*}
where $E^{(1)}$ is an index set of cardinality $|E_{n-1,d}|+1$, each $c'_j$ is a reciprocal of products of $\lambda_i(\bar{\lat L})$'s and $\|v_n\|$, so that $c'_j(H/k)^{\gamma_j}$ is invariant under the appropriate scaling. The leading error term is of degree $n-b(n-1,d)$.


\section{Error term of \eqref{eq:p1mid}}

In this section, we work on the intended error term of \eqref{eq:p1mid}, namely
\begin{equation} \label{eq:errorgoal}
\sum_{\lat B \in \Gr(\Z^{n-1}, d)} \frac{K(\lat B)^{d-i}\det(\lat B\bar{L})}{\det(\lat B\bar{L})_i}
\end{equation}
for $1 \leq i \leq d$. Rewrite \eqref{eq:errorgoal} as $1/(\|v_n\|)^{d-i}$ times
\begin{equation*}
\frac{1}{k^{d-i}}\sum_{\lat B \in \Gr(\Z^{n-1},d)} \frac{\det(\lat B\bar{L})}{\det(\lat B\bar{L})_i}\left(\frac{H^2}{\det(\lat B\bar{L})^2} - k^2\right)^{\frac{d-i}{2}}f_{H/k}(\lat B \bar L),
\end{equation*}
which we simplify and bound from above by
\begin{equation}  \label{eq:error}
(H/k)^{d-i}\sum_{\lat B \in \Gr(\bar{\lat L},d)} \frac{f_{H/k}(\lat B)}{(\det \lat B)^{d-i-1}\det \lat B_i}.
\end{equation}

Our analysis of \eqref{eq:error} depends on the ``skewness'' of $\lat B$ and $\bar{\lat L}$. We will first explain how to deal with \eqref{eq:error} in case all $\lambda_i(\bar{\lat L})$ is of size $(H/k)^{1/d}$ --- i.e. $\bar{\lat L}$ is not too skewed --- and then work out the general case.

In addition, for the rest of this section, we assume $k=1$ for simplicity. To restore the general case, one could simply replace $H$ by $H/k$.

\subsection{When $\bar{\lat L}$ is ``not skewed''}
Assume $\lambda_{n-1}(\bar{\lat L}) \leq 2^{n-1} H^{1/d}$. For each $0 \leq d' \leq d$, consider the restriction of the sum \eqref{eq:error} to those $\lat B \in \Gr(\bar{\lat L}, d)$ for which $d'$ is the lowest number such that 
\begin{equation} \label{eq:error_ns_cond}
\lambda_{d'}(\lat B) \leq 2^{d'}H^{1/d} \mbox{ and } \lambda_{d'+1}(\lat B) - \lambda_{d'}(\lat B) > 2^{d'}H^{1/d}
\end{equation}
(in fact, the former inequality follows from the latter and the minimality of $d'$), where we interpret $\lambda_0 = 0$ and $\lambda_{d+1} = \infty$. Such a sum is then bounded by a constant times
\begin{equation} \label{eq:error_noskew}
H^{(d'+1)(1-i/d)}\sum_{\lat B \in \Gr(\bar{\lat L},d)} \frac{f_{H}(\lat B)}{\left(\det \lat B_{d'}\right)^{d-i}},
\end{equation}
where the sum is over all $\lat B$ satisfying \eqref{eq:error_ns_cond}, since it follows from the Minkowski's second and \eqref{eq:error_ns_cond} that
\begin{align*}
\frac{H^{d-i}}{(\det \lat B)^{d-i-1}\det \lat B_i} & \ll \frac{H^{d-i+(d'-i)/d}}{(\det \lat B)^{d-i-1} \det \lat B_{d'}}\\
&\ll \frac{H^{d-i+(d'-i)/d - (d-d')(d-i-1)/d}}{(\det \lat B_{d'})^{d-i}} \\
&\ll \frac{H^{(d'+1)(1-i/d)}}{\left(\det \lat B_{d'}\right)^{d-i}}.
\end{align*}

The idea for bounding \eqref{eq:error_noskew} is that, because we are assuming $\lambda_{n-1}(\bar{\lat L}) \leq 2^{n-1} H^{1/d}$, we can proceed as in Section 9 of Schmidt (\cite{Sch}). The lemma below is a refinement of Lemma 6 of \cite{Sch}, so as to make explicit the dependence on the successive minima of $\bar{\lat L}$ and $\lat B$.

\begin{lemma} \label{lemma:error1}
Let $\bar{\lat L}$ be an $m$($=n-1$ in our context) dimensional lattice.
Fix a $\lat B' \in \Gr(\bar{\lat L},d')$, and let $j = d-d'$. Then the number of $\lat B \in \Gr(\bar{\lat L},d)$ such that $\lat B_{d-j} = \lat B'$, $\lambda_{d'+1}(\lat B) \gg \lambda_m(\bar{\lat L})$ and $\det \lat B \leq H$ is
\begin{equation*}
\ll \left(\frac{\det \lat B'}{\det\bar{\lat L}}\right)^j\left(\frac{H}{\det \lat B'}\right)^{m-d'},
\end{equation*}
where the implicit constant here depends only on $n$ and the implied constants on the bound relating $\lambda_{m}(\bar{\lat L})$ and $\lambda_{d'+1}(\lat B)$.
\end{lemma}
\begin{proof}
We may assume $\det \lat B' \ll H^{d'/d}$, because by Minkowski's second
\begin{equation*}
\det \lat B' \sim \lambda_1(\lat B') \ldots \lambda_{d'}(\lat B') \leq (\lambda_1(\lat B) \ldots \lambda_{d}(\lat B))^{d'/d} \ll H^{d'/d}.
\end{equation*}


We proceed by induction on $j$. Suppose first that $j=1$. Let $\pi : \spn_\R(\bar{\lat L}) \rightarrow \spn_\R(\bar{\lat L})$ be the orthogonal projection onto the orthogonal complement of $\spn_\R(\lat B')$. Then $\pi(\bar{\lat L})$ ($\cong \bar{\lat L}/\lat B'$) is a rank $m-d'$ lattice of determinant $\det\bar{\lat L}/\det \lat B'$, and $\pi(\lat B)$ is a $1$-dimensional primitive sublattice spanned by a vector whose length is $\det \lat B / \det \lat B'$. Therefore, the number of $\lat B$ is bounded by the number of primitive vectors of $\pi(\bar{\lat L})$ of length $\leq H/\det \lat B'$.

If $\mathfrak{F}$ is a fundamental domain of $\bar{\lat L}$, then $\pi(\mathfrak{F})$ is a fundamental domain of $\pi(\bar{\lat L})$. Since we can choose an $\mathfrak{F}$ of diameter $\lambda_1(\bar{\lat L}) + \ldots + \lambda_{m}(\bar{\lat L}) \leq m\lambda_{m}(\bar{\lat L})$ and $\pi$ is a contraction, $\pi(\mathfrak{F})$ has diameter $\leq m\lambda_m(\bar{\lat L})$. So the number of vectors of $\pi(\bar{\lat L})$ of length $\leq H/\det \lat B'$ is bounded by a constant times
\begin{equation} \label{eq:error1case1}
\frac{\det \lat B'}{\det\bar{\lat L}}\left(\frac{H}{\det \lat B'} + m\lambda_m(\bar{\lat L})\right)^{m-d'} \sim \frac{\det \lat B'}{\det\bar{\lat L}}\left(\frac{H}{\det \lat B'}\right)^{m-d'}.
\end{equation}
Here we used the fact that $H/\det \lat B' \geq \det \lat B / \det \lat B' \sim \lambda_{d'+1}(\lat B) \gg \lambda_m(\bar{\lat L})$.

For a general $j$, by inductive hypothesis what we need to estimate is
\begin{equation} \label{eq:temptem}
\sum_{\lat C} \left(\frac{\det \lat C}{\det\bar{L}}\right)^{j-1}\left(\frac{H}{\det \lat C}\right)^{m-1-d'}
\end{equation}
where the sum is over all $\lat C \in \Gr(\bar{\lat L}, d'+1)$ such that $\lat C_{d'} = \lat B'$ and $\lambda_{d'+1}(\lat C) \gg \lambda_m(\bar{\lat L})$. In addition, $\lat C$ must satisfy $\det \lat C \ll \det \lat B'(H/\det \lat B')^{1/j} =: h$ say, since $\lambda_{d'+1}(\lat B) \ll (H/\det \lat B')^{1/j}$.

From the (proof of) case $j=1$, the number of $\lat C$ with $\lat C'_{d'} = \lat B'$, $\lambda_{d'+1}(\lat C) \gg \lambda_m(\bar{\lat L})$, and $\det \lat C \leq t$ is
\begin{align*}
&\ll \frac{t^{m-d'}}{\det\bar{\lat L}\det \lat B'^{m-1-d'}} &\mbox{  if $\frac{t}{\det \lat B'} \geq \mathrm{(const)}\cdot \lambda_m(\bar{\lat L})$,} \\
&\ll \frac{\det \lat B'}{\det \bar{\lat L}} \lambda_m(\bar{\lat L})^{m-d'} &\mbox{  otherwise.}
\end{align*}
But $t \geq \det \lat C \sim \det \lat B' \cdot \lambda_{d'+1}(\lat C) \gg \det \lat B' \cdot \lambda_m(\bar{\lat L})$, so we may disregard the latter possibility.


Therefore, we can apply Lemma \ref{lemma:r-s}, the Riemann-Stieltjes argument in the previous section, and deduce that \eqref{eq:temptem} is bounded by a constant times
\begin{equation} \label{eq:error1int}
\int_0^{h} \frac{t^{m-d'}}{\det\bar{\lat L}\det \lat B'^{m-1-d'}} \cdot t^{-m+1+d'+j}\frac{H^{m-d'-1}}{(\det\bar{\lat L})^{j-1}}dt,
\end{equation}
which turns out to be equal to a constant times
\begin{equation*}
\left(\frac{\det \lat B'}{\det\bar{\lat L}}\right)^j\left(\frac{H}{\det \lat B'}\right)^{m-d'},
\end{equation*}
as desired.

\end{proof}

We proceed to estimating \eqref{eq:error_noskew}. Thanks to Lemma \ref{lemma:error1}, for some constant $C > 0$ depending only on $n$ such that $\det \lat B' < CH^{d'/d}$ (which exists by Minkowski's second), we can bound it by
\begin{align*}
&H^{(d'+1)(1-i/d)} \sum_{\lat B' \in \Gr(\bar{\lat L}, d')} \frac{f_{CH^{d'/d}}(\lat B')}{(\det \lat B')^{d-i}} \left(\frac{\det \lat B'}{\det\bar{\lat L}}\right)^{d-d'}\left(\frac{H}{\det \lat B'}\right)^{n-1-d'} \\
&= \frac{H^{n-(1+d')i/d}}{(\det\bar{\lat L})^{d-d'}} \sum_{\lat B' \in \Gr(\bar{\lat L}, d')} f_{CH^{d'/d}}(\lat B')(\det \lat B')^{-n+1+i}.
\end{align*}

This can be handled again as in the previous section using Lemma \ref{lemma:r-s}, yielding $|E_{n-1,d'}| + 1$ terms of $H$-degree at most $n-i/d$ satisfying all the miscellaneous conditions such as the scaling invariance.

\subsection{The skewed case}

Now assume that $0 \leq l < n-1$ is the lowest number such that
\begin{equation*}
\lambda_l(\bar{\lat L}) \leq 2^{l}H^{1/d} \mbox{ and } \lambda_{l+1}(\bar{\lat L}) - \lambda_{l}(\bar{\lat L}) > 2^{l}H^{1/d}.
\end{equation*}

As earlier, we again restrict the sum \eqref{eq:error} to those $\lat B \in \Gr(\bar{\lat L}, d)$ for which $0 \leq d' \leq d$ is the lowest number such that 
\begin{equation*}
\lambda_{d'}(\lat B) \leq 2^{d'}H^{1/d} \mbox{ and } \lambda_{d'+1}(\lat B) - \lambda_{d'}(\lat B) > 2^{d'}H^{1/d}.
\end{equation*}
Then we must have $d' \leq l$ and $\lat B_{d'} \subseteq \bar{\lat L}_l$. There is a decomposition
\begin{equation*}
\bar{\lat L} = \bar{\lat L}_l \oplus \lat M,
\end{equation*}
where $\lat M$ is an $n-1-l$ dimensional lattice chosen as follows: take a reduced basis $\{x_1, \ldots, x_{n-1}\}$ of $\bar{\lat L}$ such that $\|x_i\| \sim \lambda_i(\bar{\lat L})$ and $\spn\{x_1, \ldots, x_l\} = \bar{\lat L}_l$. Then we let $\lat M = \spn\{x_{l+1}, \ldots, x_{n-1}\}$. Also, let $\bar{\lat M}$ to be the orthogonal projection of $\lat M$ onto $\spn_\R(\bar{\lat L}_l)^\perp \subseteq \spn_\R(\bar{\lat L})$. An important fact we will use later is that $\lambda_1(\bar{\lat M}) \gg H^{1/d}$ by construction.

We further restrict \eqref{eq:error} to those $\lat B$ for which $\rk \lat B \cap \bar{\lat L}_l = r$ for a fixed $r \in \{d', \ldots, \min(l,d)\}$, and call $\lat B_{(r)} = \lat B \cap \bar{\lat L}_l$. Note that $(\lat B_{(r)})_{d'} = \lat B_{d'}$. We also let $\lat A \subseteq \bar{\lat M}$ be the projection of $\lat B$ onto $\bar{\lat M}$. Clearly $\det \lat B = \det \lat B_{(r)} \det \lat A$, and since $\det \lat A \gg H^{(d-r)/d}$ we have $\det \lat B_{(r)} \ll H^{r/d}$.

Our considerations so far lead us to bound the restriction of \eqref{eq:error} by, for some constant $C>0$,
\begin{align*}
&H^{(d'+1)(1-i/d)}\sum_{\lat B \in \Gr(\bar{\lat L},d)} \frac{f_{H}(\lat B)}{\left(\det \lat B_{d'}\right)^{d-i}} \\
&\ll H^{(d'+1)(1-i/d)}\sum_{\lat B_{(r)} \in \Gr(\bar{\lat L}_l, r)} \frac{f_{CH^{r/d}}(\lat B_{(r)})}{(\det (\lat B_{(r)})_{d'})^{d-i}} \sum_{\lat A \in \Gr(\bar{\lat M}, d-r)} f_{H/\det \lat B_{(r)}}(\lat A).
\end{align*}
Using the induction hypothesis on our main theorem, and the fact that $\lambda_1(\bar{\lat M}) \gg H^{1/d}$, we can rewrite the inner sum so that this becomes
\begin{equation*}
\ll H^{(d'+1)(1-i/d)}\sum_{\lat B_{(r)} \in \Gr(\bar{\lat L}_l, r)} \frac{f_{CH^{r/d}}(\lat B_{(r)})}{(\det (\lat B_{(r)})_{d'})^{d-i}} \sum_{\gamma \in \{\gamma_j: j \in E_{n-1-l,d-r} \} \atop \cup \{n-1-l\}} \left(\frac{H}{\det \lat B_{(r)}}\right)^\gamma H^{-\gamma(d-r)/d}.
\end{equation*}

Let us look at one $\gamma$ at a time, and consider
\begin{align*}
&H^{\frac{r}{d}\gamma + (d'+1)(1-\frac{i}{d})} \sum_{\lat B_{(r)} \in \Gr(\bar{\lat L}_l, r)} \frac{f_{CH^{r/d}}(\lat B_{(r)})}{(\det (\lat B_{(r)})_{d'})^{d-i}} \frac{1}{(\det \lat B_{(r)})^\gamma} \\
&\ll H^{\frac{d'}{d}\gamma + (d'+1)(1-\frac{i}{d})} \sum_{\lat B_{(r)} \in \Gr(\bar{\lat L}_l, r)} \frac{f_{CH^{r/d}}(\lat B_{(r)})}{(\det (\lat B_{(r)})_{d'})^{d-i+\gamma}}.
\end{align*}

By Lemma \ref{lemma:error1} as in the previous ``not skewed'' section, we obtain that this is
\begin{align}
&= H^{\frac{d'}{d}\gamma + (d'+1)(1-\frac{i}{d})} \sum_{\lat B' \in \Gr(\bar{\lat L}_l, d')} \frac{f_{\ll H^{d'/d}}(\lat B')}{(\det \lat B')^{d-i+\gamma}} \sum_{\lat B_{(r)} \in \Gr(\bar{\lat L}_l, r) \atop (\lat B_{(r)})_{d'} = \lat B'} f_{\ll H^{r/d}}(\lat B_{(r)}) \notag \\ 
&\ll H^{\frac{d'}{d}\gamma + (d'+1)(1-\frac{i}{d})} \sum_{\lat B' \in \Gr(\bar{\lat L}_l, d')} \frac{f_{\ll H^{d'/d}}(\lat B')}{(\det \lat B')^{d-i+\gamma}} \left(\frac{\det \lat B'}{\det \bar{\lat L}_l}\right)^{r-d'} \left(\frac{H^{r/d}}{\det \lat B'}\right)^{l-d'} \notag \\
&= \frac{H^{\frac{d'}{d}\gamma + \frac{r}{d}(l - d') + (d'+1)(1-\frac{i}{d})}}{(\det \bar{\lat L}_l)^{r-d'}} \sum_{\lat B' \in \Gr(\bar{\lat L}_l, d')} f_{\ll H^{d'/d}}(\lat B')(\det \lat B')^{-d+i-\gamma+r-l}. \label{eq:arf}
\end{align}

Applying Lemma \ref{lemma:r-s}, it is seen that \eqref{eq:arf} may be bounded by at most $|E_{l,r}|+1$ error terms. We need to make sure that the $H$-degree of those terms are strictly below $n$. Here we only discuss the terms of the highest degrees, as the rest can be dealt with in a similar fashion.

If $-d+i-\gamma+r \neq 0$, estimating the sum in \eqref{eq:arf} using Lemma \ref{lemma:r-s} yields a term of $H$-degree $\frac{d'}{d}(-d+i-\gamma+r)$. Therefore, the $H$-degree of \eqref{eq:arf} equals
\begin{equation*}
\frac{rl}{d} + 1 - \frac{i}{d},
\end{equation*}
which attains its maximum $n - i/d$ only if $r=d$ and $l = n-1$. But recall that we are assuming $l < n-1$.

If $-d+i-\gamma+r = 0$, the sum is of size $O(\log H)$, in which case we can say that, for a small $\eta > 0$, the $H$-degree is $\leq n - i/d - id'/d + \eta$ if $d' \neq 0$, and is $\leq n-1-i/d + \eta$ if $d' = 0$.

\subsection{The number of the error terms}

We summarize and estimate the maximum number of error terms arising from our estimate of \eqref{eq:error} so far. If $\bar{\lat L}$ is ``not skewed,'' then our estimate yielded $(d+1)E_{n-1}$ error terms, where we write
\begin{equation*}
E_{n-1} = \max_{0 \leq d' \leq n-1} |E_{n-1,d'}| + 1,
\end{equation*}
and we understand $E_{n-1,0}$ to be the empty set.

As for the skewed case, it is really $n-1$ separate cases corresponding to the parameter $0 \leq l < n-1$, and for each $l$ we obtained at most $(l+1)E_{n-1-l}E_l$ error terms. Hence, regardless of $H$ and $d$, we are able to estimate \eqref{eq:error} using at most
\begin{equation*}
\sum_{l=0}^{n-1} (l+1)E_{n-1-l}E_{l}
\end{equation*}
terms.


\section{Summary, and a proof of Theorem \ref{thm:main}}

\subsection{A polynomial expression for $P(\lat L,d,H)$}
Summing up all our work so far, we have that

\begin{equation} \label{eq:last1}
P^1(\lat L, d, H) = \sum_{k=1}^{H/\varepsilon} \left( \prod_{p | k \atop p^\alpha \| k} p^{\alpha d} \left(1 - \frac{1}{p^d}\right)\right)\left(\frac{a(n,d)}{(\det \lat L)^d}\frac{\zeta(n)}{\zeta(n-d)} H^nk^{-n} + O\left(\sum_{j \in E^{(2)}} c_j H^{\gamma_j} k^{-\gamma_j}\right)\right)
\end{equation}
where $\varepsilon = \min_{\lat B \in \mathrm{Gr}(\Z^{n-1},d)} \det(\lat B\bar{L})$, $E^{(2)}$ is an index set of cardinality at most
\begin{equation*}
E_{n-1} + d\sum_{l=0}^{n-1}(l+1)E_{n-1-l}E_l 
\end{equation*}
(collecting all error terms from the previous two sections),
and each $c_j$ is a reciprocal of products of $\lambda_i(\bar{\lat L})$'s and $\|v_n\|$ so that $c_j H^{\gamma_j}$ is scale-invariant. In this section, we will estimate the sum \eqref{eq:last1}, and then make a choice of $v_n \in \lat L$ so that the dependence on $\lambda_i(\bar{\lat L})$'s turns into dependence on $\lambda_i(\lat L)$'s. This will prove our main theorem.

We treat \eqref{eq:last1} one monomial at a time. The highest degree term contributes
\begin{equation*}
\sum_{k=1}^{H/\varepsilon} \left( \prod_{p | k \atop p^\alpha \| k} p^{\alpha d} \left(1 - \frac{1}{p^d}\right)\right)\left(\frac{a(n,d)}{(\det \lat L)^d}\frac{\zeta(n)}{\zeta(n-d)} H^nk^{-n}\right).
\end{equation*}

The corresponding infinite series, by Lemma \ref{lemma:zeta}, equals
\begin{equation*}
\frac{a(n,d)}{(\det \lat L)^d}H^n,
\end{equation*}
the desired main term. It remains to bound the tail, which we can, up to a constant factor, approximate as
\begin{equation*}
\frac{1}{(\det \lat L)^d} \sum_{k > H/\varepsilon} H^nk^{d-n},
\end{equation*}
which is of size
\begin{equation*}
\frac{H^{d+1}\varepsilon^{n-d-1}}{(\det \lat L)^d}.
\end{equation*}

We need to show that $\varepsilon^{n-d-1}/(\det \lat L)^d$ is bounded by a reciprocal of a product of $\lambda_i(\lat L)$'s. Since $\varepsilon \sim \prod_{i=1}^d \lambda_i(\bar{\lat L})$ and $\lambda_i(\bar{\lat L}) \leq \lambda_{i+1}(\lat L)$ (which can be seen by projecting a $(i+1)$-dimensional subspace of $\R^n$ onto the orthogonal complement of $v_n$),
we have $\varepsilon \ll \prod_{i=1}^d \lambda_{i+1}(\lat L)$, and thus $\varepsilon^{n-d-1} \ll \prod_{i=1}^d \lambda_{i+1}(\lat L)^{n-d-1}$. On the other hand, $(\det \lat L)^d \sim \prod_{j=1}^n \lambda_j(\lat L)^d$, which contains the factor $\prod_{j=d+2}^n \lambda_j(\lat L)$ $d$ times. For any $i \leq d + 1$, $\lambda_i(\lat L)^{n-d-1} /\prod_{j=d+2}^n \lambda_j(\lat L) \leq 1$, so $\varepsilon^{n-d-1}/(\det \lat L)^d \ll \prod_{j=1}^{d+1} \lambda_j(\lat L)^{-d}$, as desired.

We return to other monomials in \eqref{eq:last1}. For the indices $j$ with $\gamma_j > d + 1$, the sum under consideration is
\begin{equation*}
c_j H^{\gamma_j} \sum_{k=1}^{H/\varepsilon} \prod_{p | k \atop p^\alpha \| k} p^{\alpha d} \left(1 - \frac{1}{p^d}\right)k^{-\gamma_j},
\end{equation*}
which we can bound by the corresponding infinite series and apply Lemma \ref{lemma:zeta}, obtaining $O(c_j H^{\gamma_j})$. If $\gamma_j < d + 1$, the sum is of size
\begin{equation*}
c_j H^{\gamma_j}\sum_{k=1}^{H/\varepsilon}k^{d-\gamma_j} \sim \frac{c_j H^{d+1}}{\varepsilon^{d-\gamma_j+1}},
\end{equation*}
and if $\gamma_j = d + 1$, it is
\begin{equation*}
c_j H^{\gamma_j}\sum_{k=1}^{H/\varepsilon}k^{-1} \sim c_j H^{\gamma_j}\log\frac{H}{\varepsilon} \ll \frac{c_j H^{\gamma_j + \eta}}{\varepsilon^\eta}
\end{equation*}
for any $\eta > 0$. Hence, together with the expression \eqref{eq:cau-bin} of $P^2$, we conclude that

\begin{equation*}
P(\lat L, d, H) = \frac{a(n,d)}{(\det \lat L)^d}H^n + O\left(\sum_{j \in E_{n,d}} b_j H^{\gamma_j}\right)
\end{equation*}
for some index set $E_{n,d}$ of cardinality at most $dn\sum_{l=0}^{n-1}E_{n-1-l}E_l$, and where each $b_j$ is a product of reciprocals of $\lambda_i(\lat L)$'s, $\lambda_i(\bar{\lat L})$'s, and $\|v_n\|$, so that $b_j H^{\gamma_j}$ is scale-invariant.

At this point, choose $v_n$ to be one of the shortest nonzero vectors of $\lat L$. Then the following lemma shows that we can replace $\lambda_{i-1}(\bar{\lat L})$ by $\lambda_i(\lat L)$ for each $i$, so that $b_\gamma$ would depend only on $\lat L$.

\begin{lemma} \label{lemma:LLL}
Recall that $\bar{\lat L}$ is the orthogonal projection of $\lat L$ onto the complement of a vector $v_n \in \lat L$. If we choose $v_n$ to be a shortest nonzero vector of $\lat L$, then $\lambda_{i-1}(\bar{\lat L}) \sim \lambda_i(\lat L)$ for all $i = 2, \ldots, n$.
\end{lemma}
\begin{proof}
Let $\{w_1, \ldots, w_n\}$ be a reduced basis of $\lat L$ containing $v_n = w_1$. Then, writing $\bar{w}_i$ for the projection of $w_i$ to the complement of $v_n$, $\{\bar{w}_2, \ldots, \bar{w}_n\}$ is a reduced basis of $\bar{\lat L}$. Therefore, $\|w_i\| \sim \lambda_i(\lat L)$ and $\|\bar{w}_i\| \sim \lambda_{i-1}(\bar{\lat L})$.

On the other hand, by the definition of a reduced basis, $\|\bar{w}_i\|^2 = \|w_i\|^2 - \mu^2\|w_1\|^2$ for some $|\mu| \leq 1/2$. This immediately implies $\|\bar{w}_i\| \leq \|w_i\|$, and also, since $\|w_1\| \leq \|w_i\|$, we have $\|\bar{w}_i\| \gg \|w_i\|$, completing the proof.
\end{proof}

\subsection{The number of the error terms}
Let us give a quick, crude estimate of $E_n = \max_{0 \leq d \leq n} |E_{n,d}| + 1$. From the above discussion, we have
\begin{equation*}
E_n \leq n^2\sum_{l=0}^{n-1}E_{n-1-l}E_l.
\end{equation*}
Recall that $E_0 = 1$ by definition. Also, it is clear from Section 2.2 that $E_1 = 2, E_2 = 3, E_3 = 4$. We claim in general that $E_n \leq n^{3n}$ for $n \geq 2$. Indeed, the base case is obvious, and assuming the truth for the $n-1$ case, it follows from the above inequality that
\begin{equation*}
E_n \leq n^2 \cdot n \cdot n^{3(n-1)} = n^{3n}.
\end{equation*}

\subsection{The primary error term, $d \leq n/2$}
Finally, we provide an estimate on the primary error term of $P(\lat L, d, H)$, again assuming $\|v_n\| = \lambda_1(\lat L)$. We temporarily assume $d \leq n/2$, and argue the cases $d > n/2$ by duality. Tracing back our estimates so far, there are two candidates for the primary error term: one is from the estimate of the ``main part'' \eqref{eq:maingoal}, which contributes
\begin{equation} \label{eq:primerror1}
O\left(\frac{b(\bar{\lat L})}{\|v_n\|^d}H^{n-b(n-1,d)}\right)
\end{equation}
where $b(\bar{\lat L}) = b_j(\bar{\lat L})$ for $j \in E_{n-1,d}$ corresponding to the leading error term, and the other is from the estimate of the ``error part'' \eqref{eq:errorgoal} in case $i=1$, which contributes
\begin{equation*} 
O\left(\frac{H^{n-b(n,d)}}{(\det \lat L)^{d-1} \det \bar{\lat L}}\right),
\end{equation*}
but by rewriting everything in terms of $\lambda_i(\lat L)$'s with help of Lemma \ref{lemma:LLL}, we find that this is bounded by
\begin{equation} \label{eq:primerror2}
O\left(\frac{H^{n-b(n,d)}}{(\det \lat L)^{d-b(n,d)}(\det \lat L_{n-d})^{b(n,d)}}\right).
\end{equation}
The reason we use this slightly inferior bound is that this possesses a convenient symmetry under duality, as we will see below.

We claim by induction that the main error term has degree $n-b(n,d)$, and that \eqref{eq:primerror1} is no greater than \eqref{eq:primerror2}.
In the base case $n=4, d=2$, this is clear. In general, if $d = n/2$, \eqref{eq:primerror1} is of degree strictly less than $n-b(n,d)$, and we are done. If $d < n/2$, then by the fact that $\|v_n\| = \lambda_1(\lat L)$ and Lemma \ref{lemma:LLL},
\begin{equation*}
\|v_n\|^d(\det \bar{\lat L})^{d-1/d}(\det \bar{\lat L}_{n-d})^{1/d} \sim (\det \lat L)^{d-1/d}(\det \lat L_{n-d})^{1/d},
\end{equation*}
which shows that \eqref{eq:primerror1} has the same size as \eqref{eq:primerror2}, by the inductive hypothesis on $b(\bar{\lat L})$. This proves the claim.

\subsection{The primary error term, $d > n/2$}

Write $d' = n-d$ for short. By the duality theorem \eqref{eq:duality}, $P(\lat L, d, H) = P(\lat L^P, d', H')$, where $H' = H/\det \lat L$. Observe that both has the same main term, that is,
\begin{equation*}
a(n,d)\frac{H^n}{(\det \lat L)^d} = a(n,d')\frac{{H'}^n}{(\det \lat L^P)^{d'}}.
\end{equation*}
Moreover, from the previous section, $P(\lat L^P, d', H')$ has the leading error term of size
\begin{equation*}
\frac{{H'}^{n-b(n,d')}}{(\det \lat L^P)^{d'-b(n,d')}(\det (\lat L^P)_{n-d'})^{b(n,d')}},
\end{equation*}
which is equal to
\begin{align*}
&\frac{H^{n-b(n,d)}}{(\det \lat L)^{n-b(n,d)}(\det \lat L^P)^{n-d-b(n,d)}(\det (\lat L^P)_{d})^{b(n,d)}} \\
&=\frac{H^{n-b(n,d)}}{(\det \lat L)^{d-b(n,d)}(\det \lat L \det(\lat L^P)_d)^{b(n,d)}} \\
&\ll \frac{H^{n-b(n,d)}}{(\det \lat L)^{d-b(n,d)}(\det \lat L_{n-d})^{b(n,d)}},
\end{align*}
as desired.


\section{Proofs of the variants}

\subsection{Formula for $N(\lat L, d, H)$}

An asymptotic formula on $N(\lat L, d, H)$ can be derived easily from that of $P(\lat L, d, H)$ by a standard M\"obius inversion, as in Schmidt (\cite[Sections 3,4,10]{Sch}). As in \cite{Sch}, define $\sigma_d(m)$ inductively by
\begin{align*}
\sigma_1(m) &= 1, \\
\sigma_d(m) &= \sum_{r \mid n} r^{d-1}\sigma_{d-1}(m/r).
\end{align*}

It is shown in \cite[Section 3]{Sch} that $\sigma_d(m)$ equals the number of index $m$ sublattices of a rank $d$ lattice, and that
\begin{align}
\sigma_d(m) \ll (m\log\log m)^{d-1}, \notag \\
\sum_{m=1}^{\infty} \sigma_d(m)/m^n = \prod_{i=1}^d \zeta(n+1-i) \label{eq:meme}
\end{align}
for $d \leq n-1$. From the latter it follows that

\begin{align*}
N(\lat L, d, H) &= \sum_{m=1}^{H/\varepsilon} P(\lat L, d, H/m)\sigma_d(m) \\
&= \frac{a(n,d)}{(\det \lat L)^d}\sum_{m=1}^{H/\varepsilon} (H/m)^n\sigma_d(m) + O\left(\sum_{j \in E_{n,d}} \sum_{m=1}^{H/\varepsilon} b_j(\lat L)(H/m)^{\gamma_j}\sigma_d(m)\right),
\end{align*}
where $\varepsilon := \min_{\lat B \in \Gr(\lat L,d)} \det \lat B$. We handle each sum over $m$ one at a time. For the main term, we have
\begin{equation*}
\sum_{m=1}^{H/\varepsilon} (H/m)^n\sigma_d(m) = \sum_{m=1}^{\infty} (H/m)^n\sigma_d(m) - \sum_{m > H/\varepsilon} (H/m)^n\sigma_d(m).
\end{equation*}
On the right-hand side, the first sum is $H^n\prod_{i=1}^d\zeta(n+1-i)$ by \eqref{eq:meme}, which is exactly what we need. The second sum is bounded by a constant times
\begin{equation*}
\sum_{m > H/\varepsilon} m^{d-n-1+\eta}H^n \sim \frac{H^{d+\eta}}{\varepsilon^{d-n+\eta}}
\end{equation*}
for any $\eta > 0$.

In the error term, for those $j \in E_{n,d}$ with $\gamma_j > d$ we can replace the sum $\sum_{m=1}^{H/\varepsilon}$ by the infinite sum $\sum_{m=1}^\infty$, and apply \eqref{eq:meme}. For those with $\gamma_j \leq d$, we see that
\begin{equation*}
\sum_{m=1}^{H/\varepsilon} \sigma_d(m)m^{-\gamma_j} \ll \sum_{m=1}^{H/\varepsilon} m^{d-1-\gamma_j+\eta} \sim \left(\frac{H}{\varepsilon}\right)^{d-\gamma_j+\eta}
\end{equation*}
for any $\eta > 0$. If $d < n-1$, $\eta$ can be set small enough, so that the secondary term has $H$-degree $n - b(n,d)$. If $d = n-1$, the secondary term has degree $n - 1 + \eta$.

The required properties of the coefficients $b'_j(\lat L)$ can be checked straightforwardly, so we omit the proof.

\begin{remark}
One may wonder what the formula for $N(\lat L, n, H)$ would be. In this case, the skewness of $\lat L$ induces no subtlety at all, and simply
\begin{equation*}
N(\lat L, n, H) = c \cdot \left(\frac{H}{\det \lat L}\right)^n + O\left(\left(\frac{H}{\det \lat L}\right)^{n-1+\eta}\right)
\end{equation*}
for any $\eta > 0$. The proof is identical to the argument in Section 3 of \cite{Sch}; indeed, observe that $N(\lat L, n, H) = N(\Z^n, n, H)$ for any full-rank $\lat L$ of covolume 1.
\end{remark}

\subsection{Formula for $P_\lat S(\lat L, d, H)$}

Let $\lat S \subseteq \lat L$ be a sublattice of rank $e \leq n - d$. By choosing the basis $\{v_1, \ldots, v_n\}$ of $\lat L$ so that $\{v_{n-e+1}, \ldots, v_n\}$ is a basis of $\lat S$, and proceeding analogously as in Section 3 with $\lat L / \lat S$ instead of $\bar{\lat L}$, we obtain an estimate of $P_\lat S(\lat L, d, H)$ analogous to that of $P(\lat L, d, H)$ in \eqref{eq:main}, with the coefficients $b_j$ being a product of reciprocals of $\lambda_i(\lat L)$ and $\lambda_i(\lat L/\lat S)$ (here we identify $\lat L/\lat S$ with the projection of $\lat L$ onto the orthogonal complement of $\spn_\R(\lat S)$ in $\R^n$). However, the reciprocal of $\lambda_i(\lat L/\lat S)$ could be arbitrarily large, which may cause difficulties in some applications of Theorem \ref{thm:main}. For instance, suppose one wants to compute
\begin{equation*}
\sum_{\lat A, \lat B \in \Gr(\lat L,d) \atop \lat A \cap \lat B = \{0\}} f_{H_1}(\lat A)f_{H_2}(\lat B) = \sum_\lat A f_{H_1}(\lat A) \sum_{\lat B \atop \lat A \cap \lat B = \{0\}} f_{H_2}(\lat B).
\end{equation*}
Here one is eventually led to sum the multiples of the reciprocals of $\lambda_i(\lat L/\lat A)$ over sublattices $\lat A$ of height bounded by $H_1$. It seems to be a nontrivial task to show that such a sum is asymptotically small.

Fortunately, with minor modifications to our proof of Theorem \ref{thm:main}, it is possible to provide a formula for $P_\lat S(\lat L, d, H)$ independent of $\lat S$, avoiding the above complication altogether. In this section, we explain what modifications are to be made.

Consider first the base cases $d = 1$ or $n-1$. If $d = 1$, $P_\lat S(\lat L, 1, H) = P(\lat L, 1, H) - P(\lat S, 1, H)$, and bounding the contribution from $P(\lat S, 1, H)$ in terms of $\lat L$ using $\lambda_i(\lat S) \geq \lambda_i(\lat L)$ (because $\lat S \subseteq \lat L$), we obtain the same type of estimate as in \eqref{eq:case1}. In case $d = n-1$, we must have $\rk \lat S = 1$, and thus for $\lat B \in \Gr(\lat  L, n-1)$, $\lat B \cap \lat S = \{0\}$ if and only if $\lat B^\perp \cap \lat S^\perp = \{0\}$; hence the proof follows from the $d=1$ case and the duality theorem.

For other values of $d$, we proceed by induction on $n$, and split $P_\lat S = P^1_\lat S + P^2_\lat S$ as in Section 3. For $P^2_\lat S$, we simply bound it by $P^2$. As for $P^1_\lat S$, observe that, analogously to \eqref{eq:p1rewrite}, we can write
\begin{equation*}
P^1_\lat S(\lat L, d, H) = \sum_{B \in \Gamma \backslash \mathrm{Mat}^{pr}_{d \times (n-1)}(\Z)} \sum_{k \geq 1} \sum_{h} \sum_{b \in \Z^d \atop {(hB; b)\, \mathrm{prim.} \atop \mathfrak L((hB; b)L) \cap \lat S = \{0\} }} f_H\left(
\left(hB; b \right) L \right),
\end{equation*}
where $\mathfrak L(M)$ denotes the lattice spanned by the row vectors of $M$.
The idea is that the main contribution of the above sum comes from those $B$ with $\mathfrak L(B\bar{L}) \cap \bar{\lat S} = \{0\}$, where $\bar{\lat S}$ is the projection of $\lat S$ onto $\bar{\lat L}$. Since $\mathfrak L(B\bar{L}) \cap \bar{\lat S} = \{0\}$ implies $\mathfrak L((hB; b)L) \cap \lat S = \{0\}$, we may write
\begin{align*}
P^1_\lat S = P^{1,1}_\lat S + O(P^{1,2}_\lat S),
\end{align*}
where
\begin{align*}
P^{1,1}_\lat S &= \sum_{B \in \Gamma \backslash \mathrm{Mat}^{pr}_{d \times (n-1)}(\Z) \atop \mathfrak L(B\bar{L}) \cap \bar{\lat S} = \{0\}} \sum_{k \geq 1} \sum_{h} \sum_{b \in \Z^d \atop {(hB; b)\, \mathrm{prim.} }} f_H\left(
\left(hB; b \right) L \right), \\
P^{1,2}_\lat S &= \sum_{B \in \Gamma \backslash \mathrm{Mat}^{pr}_{d \times (n-1)}(\Z) \atop \mathfrak L(B\bar{L}) \cap \bar{\lat S} \neq \{0\}} \sum_{k \geq 1} \sum_{h} \sum_{b \in \Z^d \atop {(hB; b)\, \mathrm{prim.} }} f_H\left(
\left(hB; b \right) L \right).
\end{align*}
But since $P^{1,1}_\lat S + P^{1,2}_\lat S = P^1$, it also holds that
\begin{equation*}
P^1_\lat S = P^1+ O(P^{1,2}_\lat S).
\end{equation*}
Therefore all we need to show is that $P^{1,2}_\lat S$ is small, or equivalently, that $P^{1,1}_\lat S$ is close to $P^1$.

Estimating $P^{1,1}_\lat S$ amounts to considering an analogous expression to \eqref{eq:p1mid} where the sum over $\lat B$ is further restricted to those for which $\lat B \bar{L} \cap \bar{\lat S} = \{0\}$. With the same restriction added to all the subsequent computations, all of the arguments in Section 4 goes through, including Lemma \ref{lemma:r-s}, since by the induction hypothesis $P$ and $P_\lat S$ satisfy the same asymptotics on lattices of rank $\leq n-1$. As for the error terms of $P^{1,1}_\lat S$, we may simply bound them by those of $P^1$, namely \eqref{eq:error} for $i=1, \ldots, d$, and so there are no changes to make. This shows that $P^1$ and $P^{1,1}_\lat S$ satisfy the same asymptotics, and hence that $P^{1,2}_\lat S$ is bounded by terms of leading degree strictly less than $n$.

To count the number of error terms, recall that $P_\lat S = P^1 + O(P^2) + O(P^{1,2}_\lat S)$. $P^1 + O(P^2)$ have $< n^{3n}$ error terms, and since $P^{1,2}_\lat S = P^1 - P^{1,1}_\lat S$, it has at most $2n^{3n}$ error terms. Thus the number of error terms in our formula for $P_\lat S$ is no more than $3n^{3n}$.

It remains to determine the main error term. If $d \leq n/2$, the argument of Section 6.3 carries over, showing that it is the same as in Theorem \ref{thm:main}. If $d > n/2$, extend $\lat S$ to a sublattice $\lat S' \subseteq \lat L$ of rank precisely $n-d$. Then
\begin{equation*}
P_\lat{S'}(\lat L, d, H) \leq P_\lat S(\lat L, d, H) \leq P(\lat L, d, H)
\end{equation*}
on the one hand, and on the other hand,
\begin{equation*}
P_\lat{S'}(\lat L, d, H) = P_{\lat{S'}^\perp}(\lat L^P, n-d, H/\det \lat L).
\end{equation*}
Now we can argue as in Section 6.4, using duality, to show that $P_\lat{S'}$ has the leading error term of the same size, and therefore so does $P_\lat{S}$. This completes the proof of Theorem \ref{cor:2}; for $N_\lat S$, one may proceed as in the last section.

\subsection{Flag varieties of type $(e,d)$}

Let $\lat L \subseteq \R^n$ be a lattice, and let $ 1 \leq e < d < n$. Our goal is to estimate the sum
\begin{align}
&\sum_{\lat W \in \Gr(\lat L, d)} P(\lat W, e, (H/(\det \lat W)^{n-e})^{1/d}) \notag \\ 
&= \sum_{\lat W \in \Gr(\lat L, d)} \frac{a(d, e)H}{(\det \lat W)^n} + O\left(\frac{H^{1-b/d}}{(\det \lat W)^{n-b(n+d-e)/d}(\det \lat W_{d-e})^b}\right). \label{eq:flaggoal}
\end{align}
Here $b = b(d, e)$.
Note that there is also the constraint
\begin{equation} \label{eq:flagH}
(\det \lat W)^{n-e}(\det \lat W_{e})^d \leq H
\end{equation}
coming from the definition of height.

Estimating the sum over the main term is very similar to the computation in Section 4, so we will be brief. The largest term of \eqref{eq:flaggoal}, obtained by applying Lemma \ref{lemma:r-s} to $a(d,e)H/(\det \lat W)^n$, comes from the integral
\begin{align*}
\int_{\varepsilon_d}^{(H/\varepsilon_e^d)^{1/(n-e)}} \frac{a(n,d)t^n}{(\det \lat L)^d}\left(-\frac{a(d,e)H}{t^n}\right)' dt &= \frac{a(n,d)a(d,e)nH}{(\det \lat L)^d} \int_{\varepsilon_d}^{(H/\varepsilon_e^d)^{1/(n-e)}} \frac{dt}{t} \\
&= {a(n,d)a(d,e)}\frac{n}{n-e}\frac{H}{(\det \lat L)^d}\log \frac{H}{\varepsilon_e^d\varepsilon_d^{n-e}},
\end{align*}
where $\varepsilon_e = \min\{\det \lat B : \lat B \subseteq \lat L, \dim \lat B = e\}$ and likewise for $\varepsilon_d$. The smaller terms can be computed similarly, and it turns out the largest error term is of $H$-degree $1$, and the second largest is of $H$-degree $1-b(n,d)/(n-e)$, both coming from the leading error term of $P(\lat L, d, (H/\varepsilon_e^d)^{1/(n-e)})$. One obtains a total of $2|E_{n,d}| < 2E_n$ error terms from here.

The harder part of \eqref{eq:flaggoal} is the sum over the error term, namely
\begin{equation}\label{eq:flaggoal_error}
\sum_{\lat W \in \Gr(\lat L, d)} \frac{H^{1-b/d}}{(\det \lat W)^{n-b(n+d-e)/d}(\det \lat W_{d-e})^b}.
\end{equation}

To bound this, we employ our method in Section 5 above. In order to avoid repetitive and unenlightening computations, we only show how to compute the first two largest $H$-degree terms, and suppress the $\lambda_i(\lat L)$ factors from the expressions for the error terms.

As in Section 5, for each $0 \leq d' \leq d$, we restrict the sum in \eqref{eq:flaggoal_error} to those $\lat W$ for which $d'$ is the smallest number such that
\begin{equation*}
\lambda_{d'}(\lat W) \leq 2^{d'}H^{1/nd} \mbox{ and } \lambda_{d'+1}(\lat W) - \lambda_{d'}(\lat W) > 2^{d'}H^{1/nd}.
\end{equation*}

Then we can bound \eqref{eq:flaggoal_error} by
\begin{equation}\label{eq:flag_mid}
H^{1-\frac{b}{d} + \frac{b}{nd}(d'-d+e)} \sum_{\lat W' \in \Gr(\lat L, d')} \frac{1}{(\det \lat W')^b} \sum_{\lat W \in \Gr(\lat L, d) \atop \lat W_{d'} = \lat W'} \frac{1}{(\det \lat W)^{n-b(n+d-e)/d}}.
\end{equation}

In order to work on the inner sum, we first determine the range of $\det \lat W$. By Minkowski's second, we have $(\det \lat W_{d'})^{d/d'} \ll \det \lat W$. On the other hand, again by Minkowski's second we have $\det \lat W_{d'} H^{(e-d')/nd} \leq \det \lat W_{e}$, so \eqref{eq:flagH} implies
\begin{align}
& (\det \lat W)^{n-e}(\det \lat W_{d'})^d H^{(e-d')/n} \ll H  \notag \\
&\Rightarrow \det \lat W \ll H^\frac{n-e+d'}{(n-e)n}(\det \lat W_{d'})^{-\frac{d}{n-e}} \notag \\
&\Rightarrow \det \lat W_{d'} \ll H^{d'/dn}. \label{eq:flagH2}
\end{align}
In the last line, we also used $\det \lat W_{d'} \ll (\det \lat W)^{d'/d}$.

If $d' = 0$, the outer sum of \eqref{eq:flag_mid} is vacuous, and the inner sum can be computed as in the main term estimate above, yielding $O(H + H^{1-b(n+d-e)/nd} + H^{1-b(n,d)/n})$ up to lower $H$-degree terms. Similarly, we obtain the same bound in case $d' = d$. Each of these cases add at most $2E_n$ error terms to our estimate.

So assume $1 \leq d' \leq d-1$. We will apply Lemma \ref{lemma:error1}. To do so, it is required that $\lambda_{d'+1}(\lat W) \gg \lambda_n(\lat L)$, which is true provided $\lambda_n(\lat L) \ll H^{1/nd}$. Thus, assuming $H$ is sufficiently large, the inner sum of \eqref{eq:flag_mid} is bounded by a constant times
\begin{equation*}
\frac{1}{(\det \lat L)^{d-d'}(\det \lat W')^{n-d}} \int_{(\det \lat W')^{d/d'}}^{H^\frac{n-e+d'}{(n-e)n}(\det \lat W')^{-\frac{d}{n-e}}} \frac{x^{n-d'-1}}{x^{n-b(n+d-e)/d}}dx.
\end{equation*}

We divide into cases according to whether $b(n+d-e)/d - d' < 0$ or not:

\begin{enumerate}[(i)]
\item If $b(n+d-e)/d - d' < 0$, then \eqref{eq:flag_mid} is
\begin{align*}
&\ll \frac{H^{1-\frac{b}{d} + \frac{b}{nd}(d'-d+e)}}{(\det \lat L)^{d-d'}} \sum_{\lat W' \in \Gr(\lat L, d')} \frac{1}{(\det \lat W')^{n-d+b}} \int_{(\det \lat W')^{d/d'}}^\infty \frac{x^{n-d'-1}}{x^{n-b(n+d-e)/d}}dx \\
&\ll \frac{H^{1-\frac{b}{d} + \frac{b}{nd}(d'-d+e)}}{(\det \lat L)^{d-d'}} \sum_{\lat W' \in \Gr(\lat L, d')} \frac{1}{(\det \lat W')^{n-b(n+d-d'-e)/d'}}.
\end{align*}
\eqref{eq:flagH2} imposes the additional constraint $\det \lat W' \ll H^{d'/dn}$ to this sum. Therefore, by Lemma \ref{lemma:r-s}, the contribution from the largest $H$-degree term in our estimate \eqref{eq:main} of $P(\lat L, d', \mathrm{(const)}H^{d'/dn})$ to \eqref{eq:flag_mid} is
\begin{align*}
&\ll \frac{H^{1-\frac{b}{d} + \frac{b}{nd}(d'-d+e)}}{(\det \lat L)^{d}} \int_{\varepsilon_{d'}}^{H^{d'/dn}} \frac{x^{n-1}}{x^{n-b(n+d-d'-e)/d'}} dx \\
&\ll O(H + H^{1-\frac{b}{d} + \frac{b}{nd}(d'-d+e)})
\end{align*}
up to the two largest $H$-degree terms, as desired. The contributions from the smaller terms of $P(\lat L, d', \mathrm{(const)}H^{d'/dn})$ will be discussed later.

\item If $b(n+d-e)/d - d' > 0$, then we instead have
\begin{align*}
&\ll \frac{H^{1-\frac{b}{d} + \frac{b}{nd}(d'-d+e)}}{(\det \lat L)^{d-d'}} \sum_{\lat W' \in \Gr(\lat L, d')} \frac{1}{(\det \lat W')^{n-d+b}} \int_0^{H^\frac{n-e+d'}{(n-e)n}(\det \lat W')^{-\frac{d}{n-e}}} \frac{x^{n-d'-1}}{x^{n-b(n+d-e)/d}}dx \\
&\ll \frac{H^{1-\frac{d'}{n}(1-\frac{b}{d})}}{(\det \lat L)^{d-d'}} \sum_{\lat W' \in \Gr(\lat L, d')} \frac{H^{\frac{d'}{(n-e)n}(\frac{b}{d}(n+d-e)-d')}}{(\det \lat W')^{n-d+b + \frac{d}{n-e}(\frac{b}{d}(n+d-e)-d')}}.
\end{align*}

Estimating in the same manner as in Case (i) above, this is
\begin{equation*}
O(H + H^{1-\frac{d'}{n}(1-\frac{2b}{d}+\frac{d'-b}{n-e})}),
\end{equation*}
and since $1-2b/d+(d'-b)/(n-e) \geq 0$, we are done.

\item In the rare, yet possible, case that $b(n+d-e)/d - d' = 0$, we proceed similarly, and bound \eqref{eq:flag_mid} by
\begin{align*}
&\frac{H^{1-\frac{b}{d} + \frac{b}{nd}(d'-d+e)}}{(\det \lat L)^{d-d'}} \sum_{\lat W' \in \Gr(\lat L,d')} \frac{1}{(\det \lat W')^{n-d+b}} \log\frac{H^{\frac{d'}{(n-e)n}}}{(\det \lat W')^{\frac{d}{n-e}}} \\
&\ll \frac{H^{1-\frac{b}{d} + \frac{b}{nd}(d'-d+e)}}{(\det \lat L)^{d-d'}} \sum_{\lat W' \in \Gr(\lat L,d')} \frac{1}{(\det \lat W')^{n-d+b}} \log\frac{H^{\frac{1}{n}}}{(\det \lat W')^{\frac{d'}{d}}} \\
&\ll \frac{H^{1-\frac{b}{d} + \frac{b}{nd}(d'-d+e)}}{(\det \lat L)^{d}} \int_{\varepsilon_{d'}}^{H^{d'/dn}} x^{d-b-1}\log\left(\frac{H^{1/n}}{x^{d/d'}}\right) dx \\
&= \frac{H^{1-\frac{b}{d} + \frac{b}{nd}(d'-d+e)}}{(\det \lat L)^{d}} \left[ \frac{1}{d-b}\log H^{1/n} \cdot x^{d-b} - \frac{d}{d'}\left( \frac{1}{d-b} x^{d-b} \log x - \frac{1}{(d-b)^2} x^{(d-b)} \right)  \right]_{\varepsilon_{d'}}^{H^{d'/dn}} \\
&= O(H + H^{1-\frac{b}{d} + \frac{b}{nd}(d'-d+e)}\log H).
\end{align*}

\end{enumerate}

In addition, in all three cases above, the contribution from the leading error term of $P(\lat L, d', H^{d'/dn})$ is of size $O(H^{1-b(n,d')d'/dn})$, and the number of error terms obtained is at most $2E_n$.
This completes the error estimate. We note that the related computation in Thunder (\cite{Thu2}), lines 5-6 on p.185, contains a minor error: if $d-e = 1$, the integral there diverges.

In summary, we estimated \eqref{eq:flaggoal} to be
\begin{equation*}
{a(n,d)a(d,e)}\frac{n}{n-e}\frac{H}{(\det \lat L)^d}\log \frac{H}{\varepsilon_e^d\varepsilon_d^{n-e}} + O\left(\sum_{j \in E_{n,e,d}} b_j(\lat L)H^{\gamma_j}\right),
\end{equation*}
where $E_{n,e,d}$ is an index set of cardinality at most $(d+2)\cdot 2E_n \leq 3n^{3n+1}$, $b_j(\lat L)$'s are appropriate inverse products of $\lambda_i(\lat L)$'s, and the implied constant depends on $n$ only.
The largest $\gamma_j$ is $1$, and the second largest is one of
\begin{equation*}
1-\frac{b(n,d)}{n},\  1-\frac{b(d,e)(n-e)}{nd},\  \mathrm{or}\  1-\frac{1}{n}\left(1-\frac{2b(d,e)}{d} + \frac{1-b(d,e)}{n-e}\right).
\end{equation*}
When $d \leq n/2$, it is always $1-b(n,d)/n = 1-1/dn$, but otherwise both of the other two are possible.

\end{document}